\documentclass[12pt]{amsart}

\usepackage[T1]{fontenc}
\usepackage{lmodern}
\usepackage{microtype}
\usepackage{amssymb,mathtools}
\usepackage{mathrsfs}
\usepackage{enumitem,amsmath}
\setlist{topsep=3pt plus 1pt minus 1pt,itemsep=2pt plus 1pt minus 1pt,parsep=0pt,partopsep=0pt}
\usepackage[hidelinks]{hyperref}
\usepackage{aliascnt}
\usepackage[nameinlink,noabbrev]{cleveref}

\usepackage[margin=1in]{geometry}

\newtheorem{theorem}{Theorem}[section]

\newaliascnt{proposition}{theorem}
\newtheorem{proposition}[proposition]{Proposition}
\aliascntresetthe{proposition}
\newaliascnt{lemma}{theorem}
\newtheorem{lemma}[lemma]{Lemma}
\aliascntresetthe{lemma}

\newaliascnt{corollary}{theorem}
\newtheorem{corollary}[corollary]{Corollary}
\aliascntresetthe{corollary}

\theoremstyle{definition}
\newaliascnt{definition}{theorem}
\newtheorem{definition}[definition]{Definition}
\aliascntresetthe{definition}

\newaliascnt{example}{theorem}
\newtheorem{example}[example]{Example}
\aliascntresetthe{example}

\theoremstyle{remark}
\newaliascnt{remark}{theorem}
\newtheorem{remark}[remark]{Remark}
\aliascntresetthe{remark}

\newcommand{\C}{\mathcal{C}}
\newcommand{\D}{\mathcal{D}}
\newcommand{\Eclass}{\mathcal{E}}
\newcommand{\Gclass}{\mathcal{G}}
\newcommand{\I}{\mathcal{I}}
\newcommand{\N}{\mathbb{N}}
\newcommand{\R}{\mathbb{R}}
\newcommand{\F}{\mathbb{F}}
\newcommand{\eps}{\varepsilon}
\newcommand{\abs}[1]{\lvert #1\rvert}
\newcommand{\set}[1]{\{#1\}}
\newcommand{\defeq}{\mathrel{:=}}
\newcommand{\Th}{\operatorname{Th}}
\newcommand{\Aut}{\operatorname{Aut}}
\newcommand{\SU}{\operatorname{SU}}
\newcommand{\lcm}{\operatorname{lcm}}

\AtBeginDocument{%
  \setlength{\abovedisplayskip}{5pt plus 2pt minus 2pt}%
  \setlength{\belowdisplayskip}{5pt plus 2pt minus 2pt}%
  \setlength{\abovedisplayshortskip}{3pt plus 2pt minus 1pt}%
  \setlength{\belowdisplayshortskip}{3pt plus 2pt minus 1pt}%
}

\title[Asymptotic Classes via Definable Quotients]
{Asymptotic Classes via Definable Quotients}

\author{Mostafa Mirabi}
\address{The Taft School, Watertown, CT 06795, USA\\
and Wesleyan University, Middletown, CT 06459, USA}
\email{mmirabi@wesleyan.edu}
\urladdr{https://sites.google.com/site/mostafamirabi}

\date{}

\begin{document}

\begin{abstract}
In this paper, we study when uniform interpretations and definable quotients
preserve finite-dimensional asymptotic classes. Counting in the ambient
structures gives finitely many asymptotic alternatives. However, the
corresponding parameter sets need not be definable in the interpreted
language, so in general the interpreted class is only weakly asymptotic.
We define \emph{trace reflection}, a one-sided condition which ensures this
definability. We first prove a transfer theorem for interpretations with
definable selectors and give explicit bounds on the denominator. We then
use invariant profiles to prove a transfer theorem for general definable
quotients, without assuming definable choice. We define an intrinsic visible
denominator and prove that it is the least asymptotic denominator.
We also give a coprime-witness criterion for the interpretation-dependent
bound to be sharp. We prove formula-wise, syntactic, and semantic criteria
for descent, as well as composition theorems, and show that trace reflection
is strictly weaker than uniform weak bi-interpretability. Using profiles of
finite frames, we remove the noncanonical parameters used to reconstruct
finite-field coordinates. As applications, we study projective quotient
traces and show that, for every fixed $k\geq2$, the pure incidence
structures given by graphs of polynomials of degree less than $k$ over
finite fields form a full $k$-dimensional asymptotic class. The denominator
$k$ is minimal, and every member is $K_{k,n}$-free for all $n\geq2$.
\end{abstract}

\maketitle

\section{Introduction}

Let $M$ be a finite structure and let $\varphi(x;\bar y)$ be a formula.  In
an asymptotic class, the sizes of the fibers $\varphi(M;\bar b)$ have only
finitely many possible leading terms:
$$
  \abs{\varphi(M;\bar b)}
  =\mu\abs M^\alpha+o(\abs M^\alpha).
$$
The possible pairs $(\alpha,\mu)$ depend only on $\varphi$. For each pair,
the corresponding parameter tuples form a parameter-free definable set.
If we have the numerical estimates but these parameter sets are not
uniformly definable, we obtain only a \emph{weak} asymptotic class.

Finite fields provide the basic example: the Lang--Weil estimates and the
uniform counting theorem of Chatzidakis, van den Dries, and Macintyre give
the required counting laws \cite{LangWeil1954,CDM1992}.
Macpherson--Steinhorn introduced the notion of a one-dimensional
asymptotic class, and Elwes generalized it to the finite-dimensional form
used in this paper
\cite{MacphersonSteinhorn2008,Elwes2007}.  These counting properties are also
closely connected with MS-measurability and the model theory of
pseudofinite ultraproducts
\cite{ElwesMacpherson2008,GarciaMacphersonSteinhorn2015,HrushovskiWagner2008,Mirabi-MSMeas}.

\subsection*{Counting and definability under interpretations}

We ask when a uniform interpretation preserves asymptoticity. Suppose that a class $\D$ is uniformly interpreted
in an asymptotic class $\C$.  For $A\in\D$, write $M=\alpha(A)\in\C$.  An
element of $A$ is represented by an $r$-tuple from a definable set
$X_A\subseteq M^r$, possibly modulo a definable equivalence relation $E_A$.
A formula in the target language translates to an ambient formula on
representatives. Thus ambient counting gives finitely many possible
asymptotic sizes. We must also define the parameters giving each size in the language of $A$. Thus
$$
  \begin{aligned}
  \text{full asymptotic transfer}
  &=\text{ambient counting}\\
  &\quad+\text{target-definability of the counting cases}.
  \end{aligned}
$$
Ambient counting alone gives weak asymptoticity. Existing full transfer
results obtain the required target definability from a form of weak bi-interpretability;
see
\cite[Lemma~3.7]{Elwes2007}, the discussion in
\cite[Section~2]{ElwesMacpherson2008}, and
\cite[Theorem~2.5.4]{AMSW2024}.

Suppose ambient counting partitions the representative parameter tuples
into cells
$
  \theta_1(\bar v),\ldots,\theta_t(\bar v),
$
with one asymptotic label attached to each cell.  If the interpretation has
a definable selector, every target element has a distinguished
representative, and it suffices for the pullback of each $\theta_i$ to be
definable in the target.  For a general quotient, representatives of the same target tuple may
belong to different cells.

We remove this dependence on representatives by replacing the original
cells with their invariant \emph{profiles}.  The profile of a target tuple
records which ambient cells meet its coordinatewise $E_A$-class:
$$
  P(\bar b)=
  \left\{i:\text{some representative of $\bar b$ satisfies }\theta_i\right\}.
$$
The target fiber cardinality is independent of the chosen representative.
Hence all labels in an unbounded profile must agree. Indeed, different
leading exponents or coefficients cannot give the same fiber cardinality
along an unbounded sequence. This gives a well-defined asymptotic label
for the profile. We still need to show that the profile is definable in
the target language.

We introduce \emph{trace reflection}. Roughly speaking,
an interpretation is trace reflecting if the ambient relations describing
the counting cases are definable in the target.  In the selector setting, we restrict ambient formulas to the selected
representatives. In the quotient setting, we consider ambient formulas
invariant under changing representatives.
Trace reflection asks that these restricted or invariant relations be
parameter-free definable in the target.  This condition does not require an interpretation of the entire ambient
structure in the target.

The following calculation gives the denominators in our results.
Let $T=\abs M$.
If, on one cell,
$$
  \abs A=\lambda T^{e/N}+o(T^{e/N})
  \quad\text{and}\quad
  \abs{\varphi(A;\bar b)}
  =\mu T^{d/N}+o(T^{d/N}),
$$
then $d\leq e$ and
$$
  \abs{\varphi(A;\bar b)}
  =\mu\lambda^{-d/e}\abs A^{d/e}
   +o(\abs A^{d/e}).
$$
Thus target exponents are obtained by dividing ambient fiber exponents by
the exponent governing the size of the interpreted universe.

\smallskip
\noindent\textbf{Theorem A (Theorem~\ref{thm:selector-transfer}).}
\emph{Let $\C$ be an $N$-dimensional asymptotic class, and suppose that $\D$
is uniformly parameter-interpreted in $\C$ with arity $r$.  If the
interpretation has a uniform definable selector and is selector-trace
reflecting, then $\D$ is
an $N_*$-dimensional asymptotic class, where
$
  N_*=\lcm(1,2,\ldots,Nr).
$
The denominator can be sharpened by retaining only the positive exponents
that actually occur in the size of the selector.}
\smallskip

Many natural quotients, such as projective points and unordered tuples,
need not admit definable choice. We prove the main transfer theorem
without this assumption. We first count equivalence classes in the ambient
structures and then apply the profile construction.

\smallskip
\noindent\textbf{Theorem B (Theorem~\ref{thm:quotient-transfer}).}
\emph{Let $\C$ be an $N$-dimensional asymptotic class, and suppose that $\D$
is uniformly parameter-interpreted in $\C$ with arity $r$.  If the
interpretation is quotient-trace reflecting, then $\D$ is an
$L_Q$-dimensional asymptotic
class.  Here $L_Q$ is the least common multiple of the positive exponents
occurring in the ambient asymptotic sizes of the interpreted universe on
target-unbounded cells.  In particular,
$
  L_Q\mid\lcm(1,2,\ldots,Nr).
$}
\smallskip

For a fixed target formula, the proof only requires that the finitely many
profiles from one ambient counting partition descend. Full trace reflection
is a uniform hypothesis which ensures this. Without descent, the same
argument gives weak asymptoticity. We also distinguish the bound $L_Q$,
which depends on the interpretation, from the denominator realized by
target-definable families. We prove that the latter is the least possible
asymptotic denominator of the target. Finally, we prove composition results
and explain the problem caused by choosing representatives when a later
quotient interpretation uses parameters.

Trace reflection can be verified by familiar model-theoretic conditions.
Associated with the interpretation is a two-sorted pair containing the
ambient structure, the target, the interpretation parameters, and the
quotient map.  We prove a syntactic criterion using separated relative
quantifier elimination and sentence descent
(Theorem~\ref{thm:relative-qe-trace}), and a semantic criterion using pure
stable embeddedness, lifting of target automorphisms, and determination of
pair completions by target completions
(Theorem~\ref{thm:stable-embedded-descent}).  We also give versions of these conditions for the finitely many profile
traces needed for a single target formula.

We also prove a transfer theorem for structures in which a finite frame
gives uniform coordinates, but there is no canonical frame. We record the
counting labels obtained from all good frames and use them to remove the
frame parameters. We apply this argument to polynomial incidence structures.

Finally, we show that trace reflection is strictly weaker than weak
bi-interpretability:
uniform parameter-free weak bi-interpretability implies quotient-trace
reflection, but the converse fails
(Theorem~\ref{thm:strict-weak-bi}).

\subsection*{Geometric applications}

We first apply the quotient theorem to projective space, viewed as the
quotient of $F^k\setminus\{0\}$ by nonzero scalar multiplication.  In the
full quotient-trace language, these structures form a
$(k-1)$-dimensional asymptotic class, with minimal denominator $k-1$.

Our main application concerns graphs of polynomials of bounded degree.  Fix
$k\geq2$.  For a finite field $F$, let $G_k(F)$ have point sort $F^2$, curve
sort $F[X]_{<k}$, and incidence
$$
  (x,y)\ I\ f
  \quad\Longleftrightarrow\quad
  y=f(x).
$$
The field coordinates give the counting estimates. We want to show that
the corresponding parameter cells are definable in the
\emph{pure incidence language}.

Fix $k-2$ points with distinct first
coordinates and consider the curves through them. These curves form a
two-dimensional affine space. Incidence with the remaining points gives
all but $k-1$ affine directions. We define the missing directions
uniformly and complete the partial affine plane. A finite affine frame
then gives the field structure and coordinates for every point and
polynomial curve. Finally, we use framed profiles to remove the frame
parameters. We use classical affine-plane coordinatization
\cite{Dembowski1968}, and show that the reconstruction and removal of
parameters are uniform in the pure incidence language.

\smallskip
\noindent\textbf{Theorem C (Theorem~\ref{thm:full-polynomial-incidence}).}
\emph{For every fixed $k\geq2$, the pure incidence class
$
  \{G_k(F):F\text{ a finite field}\}
$
is a full $k$-dimensional asymptotic class.  The denominator $k$ is minimal,
and every member is $K_{k,n}$-free for every $n\geq2$.}
\smallskip

The forbidden-biclique statement follows because two distinct polynomials
of degree less than $k$ agree at fewer than $k$ field elements.  Minimality
is also visible directly: a curve has $\abs F$ incident points, while the
whole structure has leading size $\abs F^k$, so a curve neighborhood has
relative exponent $1/k$.  Intersections of point conditions realize all
intermediate exponents.  For comparison, the generic $K_{m,n}$-free theories studied by
Conant--Kruckman are
$\mathrm{NSOP}_1$ but not simple \cite{ConantKruckman2019}; finite Moufang
polygons provide another geometric comparison \cite{DelloStritto2011}.

Section~2 reviews finite-dimensional
asymptotic classes and tuple counting.
Sections~3--7 develop interpretations, selectors, selector transfer,
denominator refinements, and composition.  Section~8 proves quotient
transfer and the descent criteria.  Sections~9--10 compare trace reflection
with weak bi-interpretability and give strictness examples and
model-theoretic consequences.  Section~11 treats projective quotient traces,
and Section~12 develops the polynomial-incidence application.

\section{Preliminaries}

Throughout this paper, $L$ and $L'$ are countable first-order languages,
and all structures are finite unless stated otherwise. We allow languages
with finitely many sorts. Using unary predicates for the sorts gives an
equivalent one-sorted formulation.  The cardinality $\abs M$
of a many-sorted structure is the sum of the cardinalities of its sorts.

If $\varphi(\bar x;\bar y)$ is an $L$-formula, $M$ is an $L$-structure, and
$\bar a\in M^{\abs{\bar y}}$, write
$$
  \varphi(M^{\abs{\bar x}};\bar a)
  \defeq
  \set{\bar b\in M^{\abs{\bar x}}:M\models\varphi(\bar b;\bar a)}.
$$

\subsection{Uniform asymptotic notation}

\begin{definition}\label{def:uniform-o}
Let $\Phi$ be a collection of pairs $(M,\bar a)$, where $M\in\C$ and
$\bar a$ is a tuple from $M$.  For a real-valued function $f$ and a
nonnegative function $g$ on $\Phi$, we write
$$
  f(M,\bar a)=o\bigl(g(M,\bar a)\bigr)
  \quad\text{as }\abs M\to\infty\text{ on }\Phi
$$
if, for every $\eps>0$, there is $Q\in\N$ such that
$$
  \abs{f(M,\bar a)}\leq \eps g(M,\bar a)
$$
whenever $(M,\bar a)\in\Phi$ and $\abs M>Q$.
\end{definition}

\begin{remark}\label{rem:bounded-o-convention}
If the ambient cardinalities occurring on $\Phi$ are bounded, then every
little-$o$ assertion in \cref{def:uniform-o} is formally vacuous: choose the
threshold above that bound.  Thus the assertion gives no asymptotic information on such a cell.
When bounded cells occur in a target partition, we refine them by exact
fiber cardinality.  In particular, every comparison of leading exponents
below is made on a cell on which the relevant ambient or target sizes are
unbounded.
\end{remark}

We will use the following elementary facts.

\begin{lemma}\label{lem:asymptotic-calculus}
Let $\Phi$ be a cell, let $T=\abs M$, and suppose all asymptotics below are
uniform on $\Phi$.
\begin{enumerate}[label=\textup{(\roman*)}]
  \item If
  $
    f=aT^\alpha+o(T^\alpha),$ and
   $ g=bT^\beta+o(T^\beta),
  $
  where $a,b>0$, then
  $$
    fg=abT^{\alpha+\beta}+o(T^{\alpha+\beta}).
  $$
  \item If, for $1\leq i\leq t$,
  $f_i=a_iT^{\alpha_i}+o(T^{\alpha_i})
  $
  with $a_i>0$, and $\alpha=\max_i\alpha_i$, then
  $$
    \sum_{i=1}^t f_i
    =\left(\sum_{\alpha_i=\alpha}a_i\right)T^\alpha+o(T^\alpha).
  $$
  \item Suppose
  $
    s=\lambda T^a+o(T^a),$ and
  $  f=\mu T^b+o(T^b),
  $
  where $\lambda,\mu>0$ and $a>0$.  Then
  $$
    f=\mu\lambda^{-b/a}s^{b/a}+o(s^{b/a}).
  $$
  Moreover, if $s\leq T^r$ for a fixed $r$, the last estimate is uniform as
  $s\to\infty$ on $\Phi$.
\end{enumerate}
\end{lemma}

\begin{proof}
For item \textup{(i)}, we divide by $abT^{\alpha+\beta}$. The result is
the product of $f/(aT^\alpha)$ and $g/(bT^\beta)$, and both ratios converge
uniformly to $1$. For item \textup{(ii)}, we divide by $T^\alpha$. The
terms with smaller exponents converge uniformly to $0$, and the remaining
terms converge to their coefficients.

For \textup{(iii)}, write
$$
  \frac{s}{\lambda T^a}=1+\eta,
  \qquad
  \frac{f}{\mu T^b}=1+\xi,
$$
where $\eta,\xi\to0$ uniformly.  Then
$$
  \frac{f}{\mu\lambda^{-b/a}s^{b/a}}
  =\frac{1+\xi}{(1+\eta)^{b/a}}\longrightarrow1
$$
uniformly. Finally, since $s\leq T^r$, we have $T\to\infty$ whenever
$s\to\infty$. More precisely, $s>Q^r$ implies $T>Q$. Thus we may replace
ambient thresholds by thresholds in $s$.
\end{proof}

\begin{lemma}\label{lem:fubini-sum}
Let $B_{M,\bar a}$ be finite sets and let $w_{M,\bar a}(u)\geq0$ for
$u\in B_{M,\bar a}$.  Suppose, uniformly on a cell,
$
  \abs{B_{M,\bar a}}=\nu T^\beta+o(T^\beta)
$
with $\nu>0$, and
$
  w_{M,\bar a}(u)=\mu T^\alpha+o(T^\alpha)
$
uniformly in $u\in B_{M,\bar a}$ as well, with $\mu>0$.  Then
$$
  \sum_{u\in B_{M,\bar a}}w_{M,\bar a}(u)
  =\mu\nu T^{\alpha+\beta}+o(T^{\alpha+\beta}).
$$
\end{lemma}

\begin{proof}
Write $w(u)=\mu T^\alpha+r(u)$, where
$\sup_{u\in B}\abs{r(u)}=o(T^\alpha)$ uniformly.  Then
$$
  \sum_{u\in B}w(u)
  =\mu T^\alpha\abs B+\sum_{u\in B}r(u).
$$
The first term has the asserted asymptotic by
\cref{lem:asymptotic-calculus}\textup{(i)}.  Also
$$
  \left|\sum_{u\in B}r(u)\right|
  \leq \abs B\sup_{u\in B}\abs{r(u)}
  =o(T^{\alpha+\beta}),
$$
since $\abs B=O(T^\beta)$ uniformly.
\end{proof}

\subsection{\texorpdfstring{$N$}{N}-dimensional asymptotic classes}

\begin{definition}
\label{def:N-asymptotic}
Let $N\geq1$, and let $\C$ be a class of finite $L$-structures.  We say
that $\C$ is an \emph{$N$-dimensional asymptotic class} if, for every
$L$-formula $\varphi(x;\bar y)$ with one object variable $x$, there are the
following data:
\begin{enumerate}[label=\textup{(\roman*)}]
  \item a finite set
  $
    D_\varphi\subseteq
    \bigl(\set{0,\dots,N}\times\R_{>0}\bigr)\cup\set{(0,0)},
  $
  and
  \item for every $(d,\mu)\in D_\varphi$, a parameter-free $L$-formula
  $\theta_{\varphi,d,\mu}(\bar y)$,
\end{enumerate}
such that, for every $M\in\C$, the sets
$
  \theta_{\varphi,d,\mu}(M^{\abs{\bar y}})
  $ with $ (d,\mu)\in D_\varphi,
$
partition $M^{\abs{\bar y}}$, and, whenever
$M\models\theta_{\varphi,d,\mu}(\bar a)$,
$$
  \abs{\varphi(M;\bar a)}
  =\mu\abs M^{d/N}+o\bigl(\abs M^{d/N}\bigr)
$$
uniformly on the corresponding cell.  For the distinguished pair $(0,0)$,
the displayed equation means $\abs{\varphi(M;\bar a)}=o(1)$.
\end{definition}

\begin{definition}
A class is a \emph{weak $N$-dimensional asymptotic class} if the same finite
asymptotic alternatives exist, but the cells need not be uniformly
parameter-free definable.
\end{definition}

\begin{remark}\label{rem:zero-cell}
On a $(0,0)$-cell, the relevant cardinality is eventually zero, uniformly:
choose the little-$o$ threshold for $\eps=\tfrac12$.  More generally, suppose
an integer-valued function satisfies $n=\mu+o(1)$ on an unbounded cell.  Then
$\mu$ is a nonnegative integer and $n=\mu$ eventually, uniformly.  Indeed,
if $\mu\notin\mathbb Z$, then
$\operatorname{dist}(\mu,\mathbb Z)>0$, and convergence contradicts the fact
that every value of $n$ has distance at least
$\operatorname{dist}(\mu,\mathbb Z)$ from $\mu$.  Hence $\mu\in\mathbb Z$.
Since $n\geq0$, necessarily $\mu\geq0$; and, for sufficiently large
structures, $\abs{n-\mu}<\tfrac12$, which forces $n=\mu$.
\end{remark}

\begin{theorem}\label{thm:tuple-counting}
Suppose $\C$ is an $N$-dimensional asymptotic class.  For every formula
$\varphi(\bar x;\bar y)$ with $\abs{\bar x}=n$, there exist a finite set
$
  D_\varphi\subseteq
  \bigl(\set{0,\dots,Nn}\times\R_{>0}\bigr)\cup\set{(0,0)}
$
and a parameter-free definable partition of the parameter tuples such that,
on the cell labelled by $(d,\mu)$,
$$
  \abs{\varphi(M^n;\bar a)}
  =\mu\abs M^{d/N}+o\bigl(\abs M^{d/N}\bigr)
$$
uniformly.
\end{theorem}

\begin{proof}
We prove the statement by induction on $n$. The case $n=1$ follows from
\cref{def:N-asymptotic}. Suppose that the statement holds for $n$ object
variables. We consider
$
  \varphi(z,\bar x;\bar y),$ with $ \abs{\bar x}=n.
$
We apply the one-variable definition to $z$, with $(\bar x,\bar y)$ as
parameters.  We obtain finitely many pairs $(e_i,\mu_i)$ and formulas
$\chi_i(\bar x,\bar y)$, indexed by $i\in I$, whose realizations partition
all pairs $(\bar b,\bar a)$, and such that on the $i$th cell
$$
  \abs{\varphi(M;\bar b,\bar a)}
  =\mu_iT^{e_i/N}+o(T^{e_i/N}),
  \qquad T=\abs M.
$$
Here $(e_i,\mu_i)=(0,0)$ is allowed.

For each $i$, apply the induction hypothesis to the formula
$\chi_i(\bar x;\bar y)$.  Thus there are finitely many pairs
$(d_{ij},\nu_{ij})$, $j\in J_i$, and formulas $\rho_{ij}(\bar y)$ which
partition the $\bar y$-tuples, such that on $\rho_{ij}$,
$$
  \abs{\chi_i(M^n;\bar a)}
  =\nu_{ij}T^{d_{ij}/N}+o(T^{d_{ij}/N}).
$$

Let $f\in\prod_{i\in I}J_i$ be a choice function, and let
$
  \rho_f(\bar y)\defeq\bigwedge_{i\in I}\rho_{i,f(i)}(\bar y).
$
The nonempty $\rho_f$ form a finite definable partition of the parameter
space.  Fix one such cell, and let $d_i=d_{i,f(i)}$ and
$\nu_i=\nu_{i,f(i)}$. We define
$$
  S_i(M,\bar a)
  \defeq
  \set{(c,\bar b)\in M^{n+1}:
  M\models\varphi(c,\bar b;\bar a)\wedge\chi_i(\bar b,\bar a)}.
$$
The sets $S_i(M,\bar a)$ are pairwise disjoint and their union is
$\varphi(M^{n+1};\bar a)$.

If either $(e_i,\mu_i)=(0,0)$ or $(d_i,\nu_i)=(0,0)$, then $S_i$ is
eventually empty, uniformly in the parameters, by
\cref{rem:zero-cell}.  Otherwise,
\cref{lem:fubini-sum} gives
$$
  \abs{S_i(M,\bar a)}
  =\mu_i\nu_iT^{(e_i+d_i)/N}
   +o\bigl(T^{(e_i+d_i)/N}\bigr).
$$
Let
$
  h_f\defeq
  \max\set{e_i+d_i:i\in I,\ \mu_i\nu_i>0}
$
when this set is nonempty, and let
$
  \lambda_f\defeq~
  \sum\limits_{e_i+d_i=h_f} \mu_i\nu_i.
$
By \cref{lem:asymptotic-calculus}\textup{(ii)}, on $\rho_f$ we have
$$
  \abs{\varphi(M^{n+1};\bar a)}
  =\lambda_fT^{h_f/N}+o(T^{h_f/N}).
$$
If all $S_i$ are eventually empty, assign the pair $(0,0)$.  Since
$0\leq e_i\leq N$ and $0\leq d_i\leq Nn$, we have
$0\leq h_f\leq N(n+1)$.  There are only finitely many choice functions.
For each resulting pair, take the disjunction of all $\rho_f$ producing it.
This gives the required finite definable partition.
\end{proof}

\begin{lemma}\label{lem:subclass}
Every subclass of an $N$-dimensional asymptotic class, in the same language,
is again an $N$-dimensional asymptotic class.
\end{lemma}

\begin{proof}
For each formula, we restrict its finite parameter-free partition and
uniform estimates to the subclass. The same partition and uniformity
thresholds give the result.
\end{proof}

\begin{lemma}
\label{lem:finite-perturbation}
Let $N\geq1$.  Adding or removing finitely many isomorphism types from an
$N$-dimensional asymptotic class preserves the property of being an
$N$-dimensional asymptotic class.
\end{lemma}

\begin{proof}
Deletion follows from \cref{lem:subclass}.
For addition, let $\mathcal E$ be a finite collection of
structures to be added.  If $\mathcal E=\varnothing$, there is nothing to
prove.  Otherwise, fix $\varphi(x;\bar y)$.  Let
$\theta_1(\bar y),\dots,\theta_t(\bar y)$ be an asymptotic partition for
$\varphi$ on the original class.  To make these formulas pairwise disjoint in every structure, we replace
them by
$
  \theta'_1=\theta_1,
  \,\,
  \theta'_i=\theta_i\wedge\bigwedge_{j<i}\neg\theta_j
  \quad(2\leq i\leq t),
$
and let
$
  \theta'_0=\neg\bigvee_{i=1}^t\theta_i.
$
These formulas partition the parameter tuples in every structure.  On the
original class, $\theta'_i$ agrees with $\theta_i$ and $\theta'_0$ is empty.
Any discrepancy in an exceptional structure occurs only below the fixed
bound
$
  B=\max\set{\abs E:E\in\mathcal E}.
$
Let $\operatorname{Small}_B$ be the sentence asserting that the total
universe has at most $B$ elements (using the standard one-sorted coding in
the many-sorted case).  Refine each $\theta'_i$, including $\theta'_0$, by
the unbounded-size piece
$
  \theta'_i(\bar y)\wedge\neg\operatorname{Small}_B
$
and by the finitely many bounded pieces
$$
  \theta'_i(\bar y)\wedge\operatorname{Small}_B\wedge
  \operatorname{Exact}_{\varphi,k}(\bar y),
  \qquad 0\leq k\leq B.
$$
Here $\operatorname{Exact}_{\varphi,k}$ is the standard formula asserting
that exactly $k$ elements satisfy $\varphi(x;\bar y)$.
Give these bounded pieces their exact dimension-zero labels, retain the
original label of $\theta_i$ on its corresponding unbounded piece, and label
the (identically empty) unbounded piece associated with $\theta'_0$ by
$(0,0)$.  Taking the uniformity threshold larger than $B$ leaves all original
asymptotic estimates unchanged.  This proves the result. We have not assumed that finite isomorphism types
are first-order isolated; that need not hold in a countably infinite
language.
\end{proof}

\section{Uniform interpretations and selectors}

For simplicity, we write the proofs for relational languages. We may
replace function and constant symbols uniformly by their graph relations.
This is a definitional change, so the counting and definability statements
remain the same.

\begin{definition}\label{def:interpretation}
Let $\C$ be a class of finite $L$-structures and $\D$ a class of finite
$L'$-structures.  A \emph{uniform parameter-interpretation} $\I$ of $\D$
in $\C$ consists of:
\begin{enumerate}[label=\textup{(\roman*)}]
  \item an injection $\alpha:\D\to\C$;
  \item integers $r\geq1$ and $s\geq0$, and fixed $L$-formulas
  $
    \delta(\bar u;\bar z)$ and $
    \epsilon(\bar u,\bar v;\bar z),
  $
  with $\abs{\bar u}=\abs{\bar v}=r$ and $\abs{\bar z}=s$;
  \item for every relation symbol $R\in L'$, a fixed $L$-formula
  $R^{\I}$ on the appropriate number of $r$-tuples;
  \item for each $A\in\D$, a tuple $\bar c_A\in\alpha(A)^s$ and an
  $L'$-isomorphism
  $
    \iota_A:A\longrightarrow X_A/E_A,
  $
  where
  $
    X_A=\delta(\alpha(A)^r;\bar c_A)
  $
  and $E_A$ is the equivalence relation on $X_A$ defined by
  $\epsilon(\bar u,\bar v;\bar c_A)$.
\end{enumerate}
The relation formulas are required to be $E_A$-invariant on the
interpretation domain and to induce the transported $L'$-structure.  If
$s=0$, the interpretation is
\emph{parameter-free}.
\end{definition}

\begin{definition}\label{def:selector}
A uniform interpretation admits a \emph{uniform definable selector} if there
is an $L$-formula $\sigma(\bar u;\bar z)$ such that, for every $A\in\D$,
$
  S_A\defeq\sigma(\alpha(A)^r;\bar c_A)\subseteq X_A
$
meets every $E_A$-class in exactly one element.
\end{definition}

When a selector is present, we identify $A$ with $S_A$ via the bijection
sending each element to the unique selected representative of its image
under $\iota_A$.

\begin{lemma}\label{lem:translation}
Let $\I$ be a uniform parameter-interpretation of $\D$ in $\C$.  For every
$L'$-formula $\varphi(\bar x;\bar y)$ there is an $L$-formula
$
  \varphi^{\I}(\bar u_1,\dots,\bar u_{\abs{\bar x}};
  \bar v_1,\dots,\bar v_{\abs{\bar y}};\bar z)
$
such that, for every $A\in\D$, every tuple $(\bar a,\bar b)$ from $A$, and
every choice $(\widetilde{\bar a},\widetilde{\bar b})$ of representatives in
$X_A$,
$$
  A\models\varphi(\bar a;\bar b)
  \quad\Longleftrightarrow\quad
  \alpha(A)\models
  \varphi^{\I}(\widetilde{\bar a};\widetilde{\bar b};\bar c_A).
$$
The translated formula is invariant under replacing any representative by
an $E_A$-equivalent one.
\end{lemma}

\begin{proof}
We prove the statement by induction on the construction of $\varphi$.  Equality
$x_i=x_j$ is translated as
$\epsilon(\bar u_i,\bar u_j;\bar z)$.  An atomic relation
$R(x_{i_1},\dots,x_{i_t})$ is translated by the formula $R^{\I}$ given by
the interpretation.  By the definition of an interpretation, the equivalence and
$E_A$-invariance hold for atomic formulas.

Translation commutes with $\neg$, $\wedge$, and the other Boolean
connectives.  If
$$
  \varphi(\bar x;\bar y)=\exists w\,\psi(\bar x,w;\bar y),
$$
define
$$
  \varphi^{\I}
  \defeq
  \exists\bar t\,
  \bigl(\delta(\bar t;\bar z)\wedge
  \psi^{\I}(\bar u_1,\dots,\bar u_{\abs{\bar x}},\bar t;
  \bar v_1,\dots,\bar v_{\abs{\bar y}};\bar z)\bigr).
$$
An element of the quotient $X_A/E_A$ exists exactly when it has a
representative in $X_A$. Hence the required equivalence follows from the
induction hypothesis. Boolean operations and the restricted quantifier
also preserve invariance.
\end{proof}

\begin{corollary}\label{cor:cardinality-selector}
Assume $\I$ admits a uniform definable selector.  For every
$L'$-formula $\varphi(\bar x;\bar y)$, every $A\in\D$, and every parameter
tuple $\bar b$ from $A$, the product of the selected-representative maps is a
bijection
$$
  \varphi(A^{\abs{\bar x}};\bar b)
  \longrightarrow
  \set{\bar u\in S_A^{\abs{\bar x}}:
  \alpha(A)\models
  \varphi^{\I}(\bar u;\widetilde{\bar b};\bar c_A)}.
$$
Thus the two sets have the same cardinality.
\end{corollary}

\begin{proof}
The selector gives a bijection $A\to S_A$. Taking Cartesian powers and
applying \cref{lem:translation} gives the result.
\end{proof}

\begin{definition}\label{def:trace-reflection}
Suppose $\I$ is a uniform parameter-interpretation with interpretation
parameter tuple $\bar c_A$, and suppose $\I$ admits a uniform definable
selector.  We say that $\I$ is \emph{selector-trace reflecting} if, for every
$m\geq0$ and every parameter-free $L$-formula
$
  \rho(\bar v_1,\dots,\bar v_m;\bar z),
$
where each $\bar v_i$ is an $r$-tuple and $\bar z$ has the same length as
the interpretation-parameter tuple, there is a parameter-free $L'$-formula
$\rho^\downarrow(y_1,\dots,y_m)$ such that, for every $A\in\D$ and every
$\bar a\in A^m$ with selected representatives
$\widetilde{\bar a}\in S_A^m$,
$$
  A\models\rho^\downarrow(\bar a)
  \quad\Longleftrightarrow\quad
  \alpha(A)\models\rho(\widetilde{\bar a};\bar c_A).
$$
The case $m=0$ reflects ambient sentences evaluated with the distinguished
interpretation parameters $\bar c_A$.  In a parameter-free interpretation,
$\bar z$ and $\bar c_A$ are empty.
\end{definition}

\begin{remark}\label{rem:weaker-reflection}
The transfer theorem below only requires reflection of the finitely many
ambient formulas which define the asymptotic cells of translations of
$L'$-formulas, together with the cells determining the size of the selector.
We use the stronger formulation in \cref{def:trace-reflection}, which is
closed under Boolean combinations.
\end{remark}

\section{Known interpretation results}

\begin{theorem}\label{thm:known-transfer}
Let $\D$ be uniformly parameter-interpretable in an asymptotic class
$\C$.
\begin{enumerate}[label=\textup{(\roman*)}]
  \item In the classical $N$-dimensional setting, $\D$ is a weak
  asymptotic class; see \cite[Lemma~3.7]{Elwes2007}.
  \item For multidimensional asymptotic classes, an $R$-m.a.c.\ transfers
  to a weak $\operatorname{Frac}(R)$-m.a.c.
  \item Under uniform weak bi-interpretability, the transferred class is
  full rather than weak.
\end{enumerate}
The last two assertions follow from \cite[Theorem~2.5.4]{AMSW2024}.
\end{theorem}

\begin{remark}\label{rem:status}
In the next section, we give a sufficient condition for full
finite-dimensional asymptotic transfer and prove it directly. We require
descent of ambient traces, without assuming an interpretation of the
ambient structure in the target. We compare this condition with weak
bi-interpretability in \cref{prop:weak-biimplies-trace,thm:strict-weak-bi}.
\end{remark}

\section{Transfer through a trace-reflecting selector}

For a formula $\varphi(x;\bar y)$ and $k\in\N$, let
$\operatorname{Exact}_{\varphi,k}(\bar y)$ be the standard formula asserting
that exactly $k$ elements satisfy $\varphi(x;\bar y)$.  For $k=0$, this is
$\neg\exists x\,\varphi(x;\bar y)$; for $k>0$, it is expressed by $k$
pairwise distinct witnesses that exhaust the fiber.

\begin{lemma}
\label{lem:partition-reflection}
Assume that $\I$ is selector-trace reflecting.  Let
$$
  \rho_1(\bar v_1,\dots,\bar v_m;\bar z),\dots,
  \rho_t(\bar v_1,\dots,\bar v_m;\bar z)
$$
be parameter-free ambient formulas which, after substituting $\bar c_A$,
partition the tuples in $S_A^m$ for every $A\in\D$.  Then there are
parameter-free $L'$-formulas $\rho_1^\downarrow(\bar y),\dots,
\rho_t^\downarrow(\bar y)$ which partition $A^m$ and pull back the given
partition under the selected-representative bijection.
\end{lemma}

\begin{proof}
For each $\rho_i$, let its reflected formula be given by
\cref{def:trace-reflection}. The selected representatives of each target
tuple lie in exactly one ambient cell. Thus the reflected formulas are
pairwise disjoint and cover all target tuples. The equivalences in the
definition give the required pullback property.
\end{proof}

\begin{lemma}\label{lem:bounded-cell}
Let $\gamma(\bar y)$ define a cell in a class $\D$, and suppose that there is
$B\in\N$ such that $\abs A\leq B$ whenever
$A\models\exists\bar y\,\gamma(\bar y)$.  Then, for every
$\varphi(x;\bar y)$, the formulas
$
  \gamma(\bar y)\wedge\operatorname{Exact}_{\varphi,k}(\bar y),
  $ with $0\leq k\leq B,
$
partition the realizations of $\gamma$, and on the $k$th part the fiber has
exact cardinality $k$.  Thus these parts satisfy the asymptotic alternatives
$(0,k)$ for $k>0$ and $(0,0)$ for $k=0$.
\end{lemma}

\begin{proof}
Since every fiber is a subset of $A$, its size is one of
$0,1,\dots,B$. The result follows from the definition of the exact-count
formulas.
\end{proof}

\begin{lemma}
\label{lem:selector-normalization}
Let $\Phi$ be a cell of pairs $(A,\bar b)$, let $M=\alpha(A)$,
$T=\abs M$, and let $S_A\subseteq M^r$ be a selector, so that
$\abs A=\abs{S_A}\leq T^r$.  Suppose that, uniformly on $\Phi$,
\begin{align*}
  \abs A&=\lambda T^{e/N}+o(T^{e/N}),\\
  x(A,\bar b)&=\mu T^{d/N}+o(T^{d/N}),
\end{align*}
where $x(A,\bar b)$ is integer-valued, $0\leq x(A,\bar b)\leq\abs A$,
$\lambda,\mu>0$, and $0\leq d,e\leq Nr$.  If the target sizes are
unbounded on $\Phi$, then $e>0$, $d\leq e$, and
$$
  x(A,\bar b)
  =\mu\lambda^{-d/e}\abs A^{d/e}
   +o(\abs A^{d/e})
$$
uniformly as $\abs A\to\infty$ on $\Phi$.
\end{lemma}

\begin{proof}
Choose a sequence $(A_i,\bar b_i)\in\Phi$ with
$\abs{A_i}\to\infty$.  Since
$\abs{A_i}\leq\abs{M_i}^r$, we have $\abs{M_i}\to\infty$.  If $e=0$,
then $\abs{A_i}=\lambda+o(1)$ is bounded, a contradiction.  Thus $e>0$.
If $d>e$, then
$$
  \frac{x(A_i,\bar b_i)}{\abs{A_i}}
  =\frac{\mu}{\lambda}\abs{M_i}^{(d-e)/N}(1+o(1))
  \longrightarrow\infty,
$$
contrary to $x(A_i,\bar b_i)\leq\abs{A_i}$.  Hence $d\leq e$.
The normalized estimate now follows from
\cref{lem:asymptotic-calculus}\textup{(iii)}, with
$s=\abs A$, $a=e/N$, and $b=d/N$.  Its final assertion gives uniformity
with respect to the target-size threshold because $\abs A\leq T^r$.
\end{proof}

\begin{theorem}\label{thm:selector-transfer}
Let $\C$ be an $N$-dimensional asymptotic class, and let $\D$ be uniformly
parameter-interpretable in $\C$ by an interpretation of arity $r$. Assume
that the interpretation admits a uniform definable selector and is
selector-trace reflecting. Set $N_*\defeq\lcm(1,2,\dots,Nr)$. Then $\D$ is
an $N_*$-dimensional asymptotic class.
\end{theorem}

\begin{proof}
Let $\varphi(x;\bar y)$ be an $L'$-formula, where
$\bar y=(y_1,\dots,y_m)$.  For $A\in\D$, let $M=\alpha(A)$ and let
$S_A\subseteq M^r$ be the selector.  If $\widetilde{\bar b}$ denotes the
selected representatives of $\bar b\in A^m$, then
\cref{cor:cardinality-selector} gives
\begin{equation}\label{eq:cardinality-equality}
  \abs{\varphi(A;\bar b)}
  =\abs{\widehat\varphi(M^r;\widetilde{\bar b})},
\end{equation}
where
$$
  \widehat\varphi(\bar u;\bar v_1,\dots,\bar v_m;\bar z)
  \defeq
  \sigma(\bar u;\bar z)\wedge
  \bigwedge_{i=1}^m\sigma(\bar v_i;\bar z)\wedge
  \varphi^{\I}(\bar u;\bar v_1,\dots,\bar v_m;\bar z).
$$

We apply \cref{thm:tuple-counting} to the selector formula
$\sigma(\bar u;\bar z)$, with $\bar z$ as parameters.  There are
finitely many pairs $(e,\lambda)$, with $0\leq e\leq Nr$, and
parameter-free formulas $\chi_{e,\lambda}(\bar z)$ which partition the
ambient parameter tuples, such that, after substituting $\bar c_A$, on the
corresponding cell
\begin{equation}\label{eq:selector-size}
  \abs{S_A}
  =\lambda\abs M^{e/N}+o(\abs M^{e/N}).
\end{equation}
We apply the same theorem to $\widehat\varphi$, with the $r$ variables
$\bar u$ as object variables and $(\bar v_1,\dots,\bar v_m,\bar z)$ as
parameters.  We obtain finitely many pairs $(d,\mu)$,
with $0\leq d\leq Nr$, and parameter-free formulas
$\theta_{d,\mu}(\bar v_1,\dots,\bar v_m;\bar z)$ such that, after
substituting $\bar c_A$, on the corresponding cell
\begin{equation}\label{eq:fiber-ambient-size}
  \abs{\widehat\varphi(M^r;\widetilde{\bar b})}
  =\mu\abs M^{d/N}+o(\abs M^{d/N}).
\end{equation}
The zero pair $(0,0)$ is permitted in either partition.

For each quadruple
$
  \tau=(e,\lambda,d,\mu),
$
apply \cref{lem:partition-reflection} to the conjunction
$
  \chi_{e,\lambda}(\bar z)\wedge
  \theta_{d,\mu}(\bar v_1,\dots,\bar v_m;\bar z).
$
We obtain an $L'$-formula $\gamma_\tau(\bar y)$ whose realizations are
precisely those $(A,\bar b)$ for which
$
  M\models\chi_{e,\lambda}(\bar c_A)$
  and
  $M\models\theta_{d,\mu}(\widetilde{\bar b};\bar c_A).
$
The formulas $\gamma_\tau$ form a finite definable partition of all target
parameter tuples.

Call $\tau$ \emph{target-bounded} if some $B_\tau$ satisfies
$\abs A\leq B_\tau$ whenever $A\models\gamma_\tau(\bar b)$ for some
$\bar b$.  On each target-bounded cell, we refine $\gamma_\tau$ by
\cref{lem:bounded-cell}. The refined cells satisfy the required asymptotic
alternatives.

Now fix a cell $\gamma_\tau$ that is not target-bounded.  If the selector
pair were $(0,0)$, then \cref{rem:zero-cell} and
$\abs A=\abs{S_A}\leq\abs M^r$ would force the target sizes to be bounded.
Likewise, a selector pair $(0,\lambda)$ with $\lambda>0$ gives
$\abs A=\lambda+o(1)$ and hence bounded target sizes.  Thus $e>0$ and
$\lambda>0$.

If the fiber pair is $(0,0)$, then \cref{rem:zero-cell}, together with
$\abs A\leq\abs M^r$, shows that the target fiber is eventually empty,
uniformly as $\abs A\to\infty$ on this cell.  Assign the target pair $(0,0)$.
Suppose instead that $\mu>0$.  The set counted in
\eqref{eq:fiber-ambient-size} is a subset of $S_A$.  Applying
\cref{lem:selector-normalization} to
$x(A,\bar b)=\abs{\varphi(A;\bar b)}$ gives $d\leq e$ and
\begin{equation}\label{eq:target-normalized}
  \abs{\varphi(A;\bar b)}
  =\nu_\tau\abs A^{d/e}+o(\abs A^{d/e}),
  \qquad
  \nu_\tau\defeq\mu\lambda^{-d/e},
\end{equation}
uniformly as $\abs A\to\infty$ on $\gamma_\tau$.

Because $1\leq e\leq Nr$, the integer $e$ divides $N_*$.  Therefore
$
  k_\tau\defeq N_*\frac de
$
is an integer.  Since $0\leq d\leq e$, we have
$0\leq k_\tau\leq N_*$, and equation
\eqref{eq:target-normalized} takes the required form
$$
  \abs{\varphi(A;\bar b)}
  =\nu_\tau\abs A^{k_\tau/N_*}
   +o(\abs A^{k_\tau/N_*}).
$$

There are finitely many original cells and finitely many exact-count
refinements of the target-bounded cells. For each pair $(k,\nu)$, we take
the disjunction of the cells with that pair. Taking the maximum of their
finitely many thresholds gives a uniform estimate on the union. Thus we
obtain a finite parameter-free definable asymptotic partition for
$\varphi$. Since $\varphi$ was arbitrary, $\D$ is an
$N_*$-dimensional asymptotic class, as desired.
\end{proof}

\begin{corollary}\label{cor:full-induced}
In the setting of \cref{thm:selector-transfer}, suppose the $L'$-structure
on $A$ is uniformly definitionally equivalent, through the selector
bijection, to the full structure on $S_A$ induced by parameter-free ambient
formulas evaluated at $\bar c_A$.  Then $\D$ is an $N_*$-dimensional
asymptotic class.
\end{corollary}

\begin{proof}
The stated definitional equivalence is exactly selector-trace reflection.
Apply \cref{thm:selector-transfer}.
\end{proof}

\section{Scope of the transfer theorem}

The proof of \cref{thm:selector-transfer} uses selector-trace reflection only
for the following finite collection of formulas attached to a fixed target
formula $\varphi(x;\bar y)$:
\begin{enumerate}[label=\textup{(\roman*)}]
  \item the ambient formulas defining the asymptotic cells for the selector
  $\sigma(\bar u;\bar z)$; and
  \item the ambient formulas defining the asymptotic cells for the translated
  fiber formula $\widehat\varphi(\bar u;\bar v;\bar z)$.
\end{enumerate}
Hence, the full trace-reflection hypothesis can be replaced by the
weaker requirement that these joint cells have uniformly parameter-free
definable pullbacks to the target.  We retain the stronger hypothesis in
\cref{def:trace-reflection} because it is intrinsic, independent of the
chosen asymptotic partitions, and closed under Boolean combinations.

\begin{proposition}
\label{prop:relative-qe}
In the setting of \cref{thm:selector-transfer}, suppose that for every
parameter-free $L$-formula
$
  \rho(\bar v_1,\dots,\bar v_m;\bar z)
$
there is a parameter-free $L'$-formula $\rho^S(\bar y)$ such that, uniformly
for $A\in\D$ and selected representatives $\widetilde{\bar a}\in S_A^m$,
$$
  \alpha(A)\models\rho(\widetilde{\bar a};\bar c_A)
  \quad\Longleftrightarrow\quad
  A\models\rho^S(\bar a).
$$
Then the interpretation is selector-trace reflecting, and hence $\D$ is an
$N_*$-dimensional asymptotic class.
\end{proposition}

\begin{proof}
The displayed equivalence is precisely the condition in
\cref{def:trace-reflection}, so the conclusion follows from
\cref{thm:selector-transfer}.
\end{proof}

\begin{remark}
For a concrete family, it remains to prove selector-trace reflection, or
at least reflection of the finite joint cell partitions described above.
This amounts to a relative quantifier-elimination or uniform
induced-structure problem.
\end{remark}

\section{Denominator bounds and closure properties}

The denominator in \cref{thm:selector-transfer} need not be the least
possible one. The same proof gives a smaller denominator using only the
asymptotic dimensions that occur for the interpreted universe.

\begin{definition}\label{def:essential-selector}
Fix an asymptotic partition
$
  \set{\chi_{e,\lambda}(\bar z):(e,\lambda)\in\Pi_S}
$
for the selector formula $\sigma(\bar u;\bar z)$, as in
\eqref{eq:selector-size}.  A positive integer $e$ is an
\emph{essential selector dimension} if, for some $\lambda>0$, the target
structures satisfying $\chi_{e,\lambda}(\bar c_A)$ have unbounded
cardinality.  Let
$$
  \mathscr E_S\defeq
  \set{e>0:e\text{ is an essential selector dimension}}
$$
and, with the convention $\lcm(\varnothing)=1$, let
$
  L_S\defeq\lcm(\mathscr E_S).
$
\end{definition}

\begin{theorem}
\label{thm:essential-selector}
Under the hypotheses of \cref{thm:selector-transfer}, the class $\D$ is an
$L_S$-dimensional asymptotic class.  In particular,
$
  L_S\mid\lcm(1,2,\dots,Nr).
$
\end{theorem}

\begin{proof}
We use the joint asymptotic partition from the proof of
\cref{thm:selector-transfer}.  Suppose first that
$\mathscr E_S=\varnothing$.  A positive-exponent cell is target-bounded by
the definition of $\mathscr E_S$.  On a zero-exponent cell, the selector
size is either $o(1)$ or $\lambda+o(1)$, and is therefore bounded for
sufficiently large ambient structures.  Below the corresponding threshold, its size is bounded by a
fixed power of that threshold.  Hence every selector-size cell supports only
target structures of bounded size.  There are finitely many cells, so $\D$
has uniformly bounded cardinality.
For any formula $\varphi(x;\bar y)$, the exact-count formulas
$\operatorname{Exact}_{\varphi,k}(\bar y)$ for the finitely many possible
values of $k$ give a definable asymptotic partition.  Thus $\D$ is
$1=L_S$-dimensional.

Suppose $\mathscr E_S\neq\varnothing$.  On every target-unbounded joint cell,
the selector exponent $e$ belongs to $\mathscr E_S$, while
\cref{lem:selector-normalization} gives a target exponent $d/e$ with
$0\leq d\leq e$.  Since $e\mid L_S$,
$$
  \frac de=\frac{k}{L_S},
  \qquad k\defeq d\frac{L_S}{e}\in\set{0,\dots,L_S}.
$$
The target-bounded cells are handled by \cref{lem:bounded-cell}.  Reflecting
and merging the finitely many cells exactly as in the proof of
\cref{thm:selector-transfer} gives an $L_S$-dimensional asymptotic
partition for every $\varphi(x;\bar y)$.
\end{proof}

\begin{corollary}\label{cor:pure-selector-dim}
Assume the hypotheses of \cref{thm:selector-transfer}, and suppose that all
unbounded selector cells have the same positive exponent $e_0$.  Then $\D$
is an $e_0$-dimensional asymptotic class.
\end{corollary}

\begin{proof}
In this case $\mathscr E_S=\set{e_0}$, so $L_S=e_0$.  Apply
\cref{thm:essential-selector}.
\end{proof}

\begin{proposition}\label{prop:def-equivalence}
Suppose that classes $\D_1$ and $\D_2$, on the same finite underlying sets,
are uniformly parameter-free definitionally equivalent.  If $\D_1$ is an
$N$-dimensional asymptotic class, then so is $\D_2$.
\end{proposition}

\begin{proof}
Every formula of the language of $\D_2$ has a uniform parameter-free
translation into the language of $\D_1$, with the same free variables and
exactly the same fibers.  Translate the finitely many formulas defining an
asymptotic partition back to the language of $\D_2$.  Fiber cardinalities,
coefficients, exponents, and uniform error estimates are unchanged.
\end{proof}

\begin{lemma}
\label{lem:denominator-refinement}
Suppose that a class $\D$ is an $M$-dimensional asymptotic class for some
$M\geq1$ and is also a weak $N$-dimensional asymptotic class.  Then $\D$ is
an $N$-dimensional asymptotic class.
\end{lemma}

\begin{proof}
Fix a formula $\varphi(x;\bar y)$, and let
$
  \set{\theta_i(\bar y):i\in I}
$
be a parameter-free definable asymptotic partition given by the
$M$-dimensional asymptotic class.  Write the label of the $i$th
cell as $(a_i,\mu_i)$, so that its positive exponent is $a_i/M$; allow also
the zero label $(0,0)$.

Fix $i$.  Suppose first that the structures in which $\theta_i$ is realized have
unbounded cardinality.  The weak $N$-dimensional structure
gives finitely many set-theoretic alternatives
$$
  \nu_j\abs A^{b_j/N}+o(\abs A^{b_j/N})
$$
and the zero alternative.  At least one weak alternative occurs on an
unbounded sequence of pairs $(A,\bar c)$ satisfying $\theta_i(\bar c)$.
Along that sequence, the full and weak estimates apply to the same
integer-valued fiber sizes.  Comparing them shows that either both labels
are zero, or
$$
  \frac{a_i}{M}=\frac{b_j}{N}
  \quad\text{and}\quad
  \mu_i=\nu_j.
$$
Indeed, if the exponents differ, the quotient of the two main terms tends
to $0$ or $\infty$; if the exponents agree but the coefficients differ, the
quotient tends to a number other than $1$.  Thus the exponent already
attached to the definable cell $\theta_i$ lies in the $N$-grid.  The same
argument shows that every weak label occurring unboundedly on this cell has
the same exponent and coefficient.

If the structures in which $\theta_i$ is realized have bounded
cardinality, we refine $\theta_i$ by the finitely many formulas
$\operatorname{Exact}_{\varphi,k}(\bar y)$ which can occur there.  These
refined cells have exact dimension-zero labels.  We do this for each of the finitely many $\theta_i$ and relabel the
unbounded cells by their corresponding $N$-grid labels. This gives a
parameter-free definable $N$-dimensional asymptotic partition for
$\varphi$. Since it is enough to consider one object variable, $\D$ is
$N$-dimensional.
\end{proof}

\begin{lemma}
\label{lem:naming-parameters}
Let $\C$ be an $N$-dimensional asymptotic class and let
$\eta(\bar z)$ be a parameter-free formula.  Expand the language by constants
$\bar c$ of the sorts of $\bar z$, and let
$$
  \C_\eta
  =\bigg\{(M,\bar c):M\in\C,\ M\models\eta(\bar c)\bigg\}.
$$
Then $\C_\eta$ is an $N$-dimensional asymptotic class.
\end{lemma}

\begin{proof}
An expanded-language formula $\varphi(x;\bar y,\bar c)$ is the same formula
in the original language with $(\bar y,\bar c)$ treated as its parameter
tuple.  Apply the asymptotic-class axiom in $\C$.  Each original cell formula
$\theta(\bar y,\bar z)$ becomes the parameter-free expanded-language formula
$\theta(\bar y,\bar c)$, and the fiber cardinalities and ambient structure
sizes are unchanged.  Restricting to the definable condition
$\eta(\bar c)$ does not affect either the partition or its uniformity.
\end{proof}

\begin{theorem}
\label{thm:composition}
Let $\I$ uniformly parameter-interpret $\D$ in $\C$ with arity $r$, and let
$\mathcal J$ uniformly parameter-interpret $\Eclass$ in $\D$ with arity
$q$.  If both interpretations admit uniform definable selectors and are
selector-trace reflecting, then the composite interpretation of $\Eclass$
in $\C$ has arity $rq$, admits a uniform definable selector, and is
selector-trace reflecting.
\end{theorem}

\begin{proof}
Let $B\in\Eclass$, let $A$ be the member of $\D$ associated with $B$ by
$\mathcal J$, and let $M$ be the member of $\C$ associated with $A$ by
$\I$.  A representative of an element of $B$ under $\mathcal J$ is a
$q$-tuple of elements of $A$.  Replace each of its coordinates by its unique
$\I$-selected representative in $M^r$.  This gives an $rq$-tuple from $M$.
Translate the domain, equivalence-relation, and relation formulas of
$\mathcal J$ through $\I$.  The resulting formulas, together with the
$\I$-selector conditions on every $r$-block, define the composite
interpretation.

Let $\bar d_B$ be the interpretation-parameter tuple for $\mathcal J$, and
take the $\I$-selected representatives of its coordinates.  The composite
interpretation parameters are the concatenation of these representatives
with the interpretation parameters for $\I$.  The formula obtained by
translating the $\mathcal J$-selector through $\I$ and restricting all
blocks to the $\I$-selector meets every composite equivalence class exactly
once.  Hence it is a uniform selector.

It remains to verify trace reflection.  Let $\rho$ be an ambient $L$-formula
on composite selected representatives and the composite interpretation
parameters.  Regard every $r$-block, including the blocks representing
$\bar d_B$, as a variable ranging over the $\I$-selector.  Trace reflection
for $\I$ turns $\rho$ into an $L'$-formula on the corresponding elements of
$A$.  Regroup its free variables into the $q$-tuples selected by
$\mathcal J$, with $\bar d_B$ occupying the distinguished parameter
positions for $\mathcal J$.  Trace reflection for $\mathcal J$ then turns
this formula into a parameter-free formula in the language of $\Eclass$.
Combining the two defining equivalences proves trace reflection for the
composite interpretation.
\end{proof}

\begin{corollary}\label{cor:composition-transfer}
If, in \cref{thm:composition}, $\C$ is $N$-dimensional, then $\Eclass$ is an
asymptotic class.  One universal denominator is
$
  \lcm(1,2,\dots,Nrq).
$
\end{corollary}

\begin{proof}
Apply \cref{thm:composition} followed by
\cref{thm:selector-transfer}.
\end{proof}

\section{Transfer on definable quotients}

We prove a transfer theorem without a definable selector.
We need the ambient formulas defining the asymptotic cells to descend to
the interpreted quotient. The proof has two steps. First, we count
equivalence classes in the ambient structures. Then we replace the
resulting parameter partition by one that is invariant under changing
representatives.

\begin{definition}
\label{def:quotient-trace}
Let $\I$ be a uniform parameter-interpretation as in
\cref{def:interpretation}.  We say that $\I$ is
\emph{quotient-trace reflecting} if the following holds.  For every $m\geq0$
and every parameter-free $L$-formula
$
  \rho(\bar v_1,\dots,\bar v_m;\bar z)
$
which is $E_A$-invariant in each $\bar v_i$ on $X_A^m$, uniformly for
$A\in\D$, there is a parameter-free $L'$-formula
$\rho^\downarrow(y_1,\dots,y_m)$ such that, for arbitrary representatives
$\bar v_i\in X_A$ of $a_i\in A$,
$$
  A\models\rho^\downarrow(a_1,\dots,a_m)
  \quad\Longleftrightarrow\quad
  \alpha(A)\models
  \rho(\bar v_1,\dots,\bar v_m;\bar c_A).
$$
As before, $m=0$ includes ambient sentences evaluated at the interpretation
parameters.
\end{definition}

For tuples $\bar v=(\bar v_1,\dots,\bar v_m)$ and
$\bar w=(\bar w_1,\dots,\bar w_m)$ from $X_A^m$, write
$$
  E_A^{(m)}(\bar v,\bar w)
  \quad\text{for}\quad
  \bigwedge_{j=1}^m E_A(\bar v_j,\bar w_j).
$$
For $m=0$, this is the trivial relation on the singleton empty tuple.

\begin{lemma}
\label{lem:quotient-profiles}
Let $\theta_i(\bar v;\bar z)$, $i\in I$, be finitely many ambient formulas
which, after substituting $\bar c_A$, partition $X_A^m$.  For every nonempty
$P\subseteq I$ there is an $E_A^{(m)}$-invariant ambient formula
$\pi_P(\bar v;\bar z)$ saying that the $E_A^{(m)}$-class of $\bar v$ meets
exactly the cells with labels in $P$.  The formulas $\pi_P$ partition
$X_A^m$ into invariant pieces.
\end{lemma}

\begin{proof}
Let $E^{(m)}(\bar v,\bar w;\bar z)$ be the coordinatewise formula induced
by the interpretation formula $\epsilon$.  Let
\begin{align*}
  \pi_P(\bar v;\bar z)
  \defeq{}&
  \bigwedge_{i\in P}
  \exists\bar w\,
  \bigl(E^{(m)}(\bar v,\bar w;\bar z)
  \wedge\theta_i(\bar w;\bar z)\bigr)\\
  &\wedge
  \bigwedge_{i\notin P}
  \neg\exists\bar w\,
  \bigl(E^{(m)}(\bar v,\bar w;\bar z)
  \wedge\theta_i(\bar w;\bar z)\bigr),
\end{align*}
where every quantified $r$-tuple is restricted to the interpretation
domain.  This formula depends only on the coordinatewise equivalence class
of $\bar v$.  Since the $\theta_i$ form a partition, every class has one
nonempty profile, and distinct profiles are disjoint.
\end{proof}

\begin{lemma}
\label{lem:quotient-profile-compatibility}
Let $f_A:X_A^m\to\N$ be invariant under $E_A^{(m)}$.  Suppose the formulas
$\theta_i$, $i\in I$, partition $X_A^m$, and on the $i$th cell we have,
uniformly,
$
  f_A(\bar v)
  =\mu_iT^{d_i/N}+o(T^{d_i/N}),
  $ \, $ T=\abs{\alpha(A)},
$
where either $\mu_i>0$ and $d_i\geq0$, or the label is $(0,0)$ and the right
side means $o(1)$.  Fix a profile $P\subseteq I$.  On any subfamily of that
profile for which the ambient cardinalities $T$ are unbounded, either all
labels in $P$ are $(0,0)$, or all labels in $P$ are the same positive
pair $(d,\mu)$.
\end{lemma}

\begin{proof}
Choose a sequence in the subfamily with $T\to\infty$.  For each $i\in P$
and each structure in the sequence, choose a representative in the $i$th
cell.  Within each structure, invariance of $f_A$ gives a common integer
value, denoted by $f$, for all these representatives.

A zero label forces $f=o(1)$ and therefore $f=0$ eventually.  This is
incompatible with a positive label, which gives
$f/T^{d/N}\to\mu>0$.  If $i,j\in P$ have positive labels, then the same
numbers $f$ satisfy
$
  \frac{f}{T^{d_i/N}}\longrightarrow\mu_i,
  $\, $
  \frac{f}{T^{d_j/N}}\longrightarrow\mu_j.
$
Different exponents would make the quotient of the left sides tend to zero
or infinity.  Thus $d_i=d_j$, and then $\mu_i=\mu_j$.
\end{proof}

\begin{lemma}
\label{lem:uniform-class-quotient}
For each parameter tuple in a fixed cell, let a finite set $Y$ be partitioned
into $q$ nonempty classes.  Let $T=\abs M$.  Suppose, uniformly on the cell,
$$
  \abs Y=\lambda T^{a/N}+o(T^{a/N}),
  \qquad \lambda>0,
$$
and suppose that every class $C\subseteq Y$ satisfies, uniformly in $C$ as
well,
$$
  \abs C=\nu T^{e/N}+o(T^{e/N}),
  \qquad \nu>0.
$$
Assume in addition that the ambient cardinalities $T$ occurring on the cell
are unbounded.  Then $a\geq e$ and
$$
  q=\frac\lambda\nu T^{(a-e)/N}
    +o(T^{(a-e)/N})
$$
uniformly.
\end{lemma}

\begin{proof}
Since $T$ is unbounded on the cell and $\lambda>0$, the set $Y$ is
nonempty for every sufficiently large value of $T$ occurring on that cell.
If $a<e$, then for any class $C\subseteq Y$ one would have
$\abs Y/\abs C\to0$ uniformly along those values, contradicting
$\abs C\leq\abs Y$.  Hence $a\geq e$.

Fix $0<\eta<\nu$.  For all sufficiently large $T$, uniformly on the cell
and in the classes,
$$
  (\nu-\eta)T^{e/N}\leq\abs C\leq
  (\nu+\eta)T^{e/N}.
$$
Hence,
$$
  \frac{\abs Y}{(\nu+\eta)T^{e/N}}
  \leq q\leq
  \frac{\abs Y}{(\nu-\eta)T^{e/N}}.
$$
Divide the displayed inequalities by $T^{(a-e)/N}$, use the uniform
asymptotic for $\abs Y$, and then let $\eta\downarrow0$.  This proves
uniform convergence to $\lambda/\nu$.
\end{proof}

\begin{lemma}
\label{lem:ambient-quotient-counting}
Let $\C$ be an $N$-dimensional asymptotic class.  Let
$X_{\bar c}\subseteq M^r$ be uniformly definable, and let
$E_{\bar c}$ be a uniformly definable equivalence relation on
$X_{\bar c}$.  Suppose that
$$
  \psi(\bar u;\bar v_1,\dots,\bar v_m;\bar z)
$$
has its $\bar u$-fiber contained in $X_{\bar c}$ and equal to a union of
$E_{\bar c}$-classes whenever the parameter tuples are valid.  Then there
are finitely many labels
$$
  (d_i,\mu_i)\in
  \bigl(\set{0,\dots,Nr}\times\R_{>0}\bigr)\cup\set{(0,0)}
$$
and parameter-free ambient formulas
$\theta_i(\bar v_1,\dots,\bar v_m;\bar z)$ which partition the valid
representative parameter tuples and such that, uniformly on the $i$th cell,
$$
  \left|
    \set{[\bar u]_{E_{\bar c}}:
      M\models\psi(\bar u;\bar v_1,\dots,\bar v_m;\bar c)}
  \right|
  =\mu_i\abs M^{d_i/N}+o(\abs M^{d_i/N}),
$$
with the usual interpretation of the zero label.  The formulas $\theta_i$
are not asserted to be invariant under changing the representative
parameters.
\end{lemma}

\begin{proof}
We use the quotient argument of Elwes \cite[Lemma~3.4]{Elwes2007}, with a
refinement of the parameter partition. Let $T=\abs M$. We first count
the individual equivalence classes. Apply \cref{thm:tuple-counting} to
$$
  K(\bar w;\bar u,\bar z)
  \defeq
  \delta(\bar w;\bar z)\wedge\delta(\bar u;\bar z)
  \wedge\epsilon(\bar w,\bar u;\bar z),
$$
with the $r$ variables $\bar w$ as object variables.  We obtain a finite
partition $\rho_j(\bar u;\bar z)$, $j\in J$, of $X_{\bar c}$ and labels
$(e_j,\nu_j)$, where $0\leq e_j\leq Nr$, $\nu_j>0$, or the label is
$(0,0)$.

Replace this partition by the invariant profiles $\pi_P(\bar u;\bar z)$
from \cref{lem:quotient-profiles}.  Call a profile $P$ \emph{active} if it
is realized in valid interpretation data with arbitrarily large ambient
cardinality.  There are only finitely many profiles.  On an active profile,
\cref{lem:quotient-profile-compatibility}, applied to the size of the
$E_{\bar c}$-class of $\bar u$, shows that all labels in $P$ are the same
positive pair $(e_P,\nu_P)$.  A zero pair is impossible because every class
is nonempty.  Moreover, this estimate is uniform over all classes with
profile $P$.  Fix $j\in P$ and choose, in every such class, a
representative in the $j$th cell.  Because there are only finitely many
profiles, a single ambient-size threshold excludes all inactive profiles.
Hence inactive profiles contribute nothing once $T$ is sufficiently large.

For every active profile $P$, define
$$
  Y_P(\bar u;\bar v_1,\dots,\bar v_m;\bar z)
  \defeq
  \psi(\bar u;\bar v_1,\dots,\bar v_m;\bar z)
  \wedge\pi_P(\bar u;\bar z).
$$
This is a union of $E_{\bar c}$-classes.  Apply
\cref{thm:tuple-counting} to $Y_P$, again with the $r$ variables $\bar u$
as object variables.  It gives finitely many parameter cells and, on each,
either
$
  \abs{Y_P}=o(1)
$
or
$$
  \abs{Y_P}=\lambda_PT^{a_P/N}+o(T^{a_P/N}),
  \qquad \lambda_P>0,
$$
where $0\leq a_P\leq Nr$.  Take the common finite refinement of these
partitions over all active profiles.

Fix a cell of this joint refinement.  If the ambient cardinalities on the
cell are bounded, assign the final quotient count the label $(0,0)$.  This
is a valid uniform $o(1)$ assertion on a bounded cell, and no asymptotic
calculation is needed there.

Now suppose that the ambient cardinalities on the joint cell are unbounded.
A zero-labelled $Y_P$ is eventually empty.  For a positive-labelled $Y_P$,
every equivalence class in it has the uniform size attached to $P$, so the
unboundedness hypothesis of \cref{lem:uniform-class-quotient} is satisfied,
and that lemma gives
$$
  \abs{Y_P/E_{\bar c}}
  =\frac{\lambda_P}{\nu_P}
   T^{(a_P-e_P)/N}
   +o(T^{(a_P-e_P)/N}),
$$
with $0\leq a_P-e_P\leq Nr$.  The quotient set defined by $\psi$ is the
disjoint union of the finitely many quotient strata $Y_P/E_{\bar c}$.
By \cref{lem:asymptotic-calculus}\textup{(ii)}, its cardinality has leading
exponent
$$
  d=\max_P(a_P-e_P)
$$
and leading coefficient equal to the sum of
$\lambda_P/\nu_P$ over the profiles attaining that exponent.  If every
stratum is eventually empty, use the label $(0,0)$.

There are finitely many joint cells and finitely many resulting pairs.
For each pair, we take the disjunction of its joint parameter cells. This
gives the required parameter-free formulas $\theta_i$. All estimates above
are uniform, and taking finite sums and finite unions preserves uniformity.
\end{proof}

\begin{lemma}
\label{lem:quotient-normalization}
Let $M=\alpha(A)$, let $T=\abs M$, and suppose
$\abs A\leq T^r$.  Suppose that, uniformly on a cell,
\begin{align*}
  \abs A&=\lambda T^{e/N}+o(T^{e/N}),\\
  x(A,\bar b)&=\mu T^{d/N}+o(T^{d/N}),
\end{align*}
where $\lambda,\mu>0$, $x$ is integer-valued, and
$0\leq x(A,\bar b)\leq\abs A$.  If the target sizes are unbounded on the
cell, then $e>0$, $d\leq e$, and
$$
  x(A,\bar b)
  =\mu\lambda^{-d/e}\abs A^{d/e}
   +o(\abs A^{d/e})
$$
uniformly as $\abs A\to\infty$.
\end{lemma}

\begin{proof}
We follow the proof of \cref{lem:selector-normalization}, using the
assumed bound $\abs A\leq T^r$.  Unbounded target size
forces $T\to\infty$ and excludes $e=0$.  If $d>e$, the ratio
$x(A,\bar b)/\abs A$ tends to infinity, contradicting $x\leq\abs A$.
The normalized estimate follows from
\cref{lem:asymptotic-calculus}\textup{(iii)}, and its final assertion turns
ambient thresholds into target-size thresholds.
\end{proof}

\begin{definition}
\label{def:essential-quotient}
Apply \cref{lem:ambient-quotient-counting} to the formula defining the whole
interpreted universe $X_A/E_A$.  Since there are no representative
parameters, fix the resulting ambient partition
$$
  \set{\chi_{e,\lambda}(\bar z):(e,\lambda)\in\Pi_U}.
$$
Let $\mathscr E_Q$ be the set of positive $e$ for which, for some
$\lambda>0$, the target structures on the cell
$\chi_{e,\lambda}(\bar c_A)$ have unbounded cardinality, and let
$
  L_Q\defeq\lcm(\mathscr E_Q),
$
with $\lcm(\varnothing)=1$.
\end{definition}

\begin{theorem}
\label{thm:quotient-transfer}
Let $\C$ be an $N$-dimensional asymptotic class, and let $\D$ be uniformly
parameter-interpretable in $\C$ by an interpretation of arity $r$.  If the
interpretation is quotient-trace reflecting, then $\D$ is an
$L_Q$-dimensional asymptotic class.  In particular, it is an
$\lcm(1,2,\dots,Nr)$-dimensional asymptotic class.
\end{theorem}

\begin{proof}
Let $\varphi(x;\bar y)$ be an $L'$-formula, with
$\bar y=(y_1,\dots,y_m)$. We translate it using \cref{lem:translation} and let
$$
  \psi(\bar u;\bar v_1,\dots,\bar v_m;\bar z)
  \defeq
  \delta(\bar u;\bar z)\wedge
  \bigwedge_{i=1}^m\delta(\bar v_i;\bar z)\wedge
  \varphi^{\I}(\bar u;\bar v_1,\dots,\bar v_m;\bar z).
$$
Its $\bar u$-fiber is a union of $E_A$-classes, and its quotient
cardinality is exactly $\abs{\varphi(A;\bar b)}$.  Apply
\cref{lem:ambient-quotient-counting}.  We obtain a finite ambient partition
$\theta_i(\bar v_1,\dots,\bar v_m;\bar z)$, $i\in I$, with quotient-count
labels $(d_i,\mu_i)$.

The cells $\theta_i$ need not be invariant under changing representatives.
Replace them by the invariant profile partition $\pi_P$ from
\cref{lem:quotient-profiles}.  Apply
\cref{lem:ambient-quotient-counting} also to the whole interpreted universe,
using the fixed cells $\chi_{e,\lambda}$ from
\cref{def:essential-quotient}.  On such a cell,
\begin{equation}\label{eq:quotient-universe-size}
  \abs A=\lambda\abs M^{e/N}+o(\abs M^{e/N}),
  \qquad M=\alpha(A).
\end{equation}

Every conjunction
$
  \chi_{e,\lambda}(\bar z)\wedge
  \pi_P(\bar v_1,\dots,\bar v_m;\bar z)
$
is invariant in every representative coordinate. By quotient-trace
reflection, a parameter-free target formula
$\gamma_{e,\lambda,P}(\bar y)$ defines its pullback. These formulas form
a finite partition of all target parameter tuples.

Apply \cref{lem:bounded-cell} to each target-bounded cell, and fix one of
the remaining target-unbounded cells.  Its universe label is positive, with
$e\in\mathscr E_Q$ and $e>0$.  The target fiber cardinality is invariant under $E_A^{(m)}$, so
\cref{lem:quotient-profile-compatibility} shows that either all labels in
$P$ are $(0,0)$, or they are all the same positive pair $(d,\mu)$.

In the first case, the target fiber is eventually empty, uniformly on the
cell, and we assign $(0,0)$.  In the second case, choose any $i\in P$.
Every parameter tuple in the profile has an equivalent representative tuple in the $i$th
cell, and the quotient fiber is unchanged.  Hence, uniformly on the whole
profile,
\begin{equation}\label{eq:quotient-fiber-size}
  \abs{\varphi(A;\bar b)}
  =\mu\abs M^{d/N}+o(\abs M^{d/N}).
\end{equation}
Since the fiber is a subset of $A$, apply
\cref{lem:quotient-normalization} to
\eqref{eq:quotient-universe-size} and
\eqref{eq:quotient-fiber-size}.  We obtain
$d\leq e$ and
$$
  \abs{\varphi(A;\bar b)}
  =\mu\lambda^{-d/e}\abs A^{d/e}
   +o(\abs A^{d/e})
$$
uniformly as $\abs A\to\infty$ on the target cell.

Because $e\mid L_Q$ we have
$$
  \frac de=\frac{k}{L_Q},
  \qquad
  k=d\frac{L_Q}{e}\in\set{0,\dots,L_Q}.
$$
Thus every target-unbounded cell has an allowed $L_Q$-dimensional
alternative, while the target-bounded cells have exact dimension-zero
alternatives.  Merge cells with the same target label and take the maximum
of their finitely many uniformity thresholds.  This proves that $\D$ is
$L_Q$-dimensional.

Finally, $\mathscr E_Q\subseteq\set{1,\dots,Nr}$, so
$L_Q\mid\lcm(1,2,\dots,Nr)$.  An $L_Q$-dimensional class is also an
$hL_Q$-dimensional class for every positive integer $h$.
\end{proof}

\begin{corollary}
\label{cor:weak-quotient-transfer}
Let $\C$ be an $N$-dimensional asymptotic class, and let $\D$ be uniformly
parameter-interpretable in $\C$ by an interpretation of arity $r$.  No
trace-reflection hypothesis is imposed.  Then $\D$ is a weak
$L_Q$-dimensional asymptotic class.  In particular, it is a weak
$\lcm(1,2,\dots,Nr)$-dimensional asymptotic class.
\end{corollary}

\begin{proof}
Retain the ambient quotient-universe cells, the quotient-counting cells for
a fixed target formula, and their invariant profiles from the proof of
\cref{thm:quotient-transfer}.  For a target parameter tuple, membership in a
profile is independent of the chosen representative tuple.  Hence the
conjunction of its quotient-universe cell and its profile determines a
well-defined set-theoretic cell of target parameter tuples, even when no
formula in the target language defines that cell.

These finitely many cells partition the target parameter space. We refine
the target-bounded cells by exact fiber cardinality.  On every
unbounded cell, \cref{lem:quotient-profile-compatibility} gives one ambient
fiber label, and \cref{lem:quotient-normalization} converts it to an
asymptotic in powers of the target size.  The denominator calculation in
the proof of \cref{thm:quotient-transfer} is unchanged.  Thus all required
asymptotic alternatives exist with denominator $L_Q$; only the definability
of the cells is absent.
\end{proof}

\begin{corollary}
\label{cor:profile-descent-criterion}
In \cref{thm:quotient-transfer}, full quotient-trace reflection may be
replaced by the following weaker hypotheses.
\begin{enumerate}[label=\textup{(\roman*)}]
  \item Every cell in the fixed ambient quotient-universe partition from
  \cref{def:essential-quotient} has a parameter-free pullback to the target
  class.
  \item For each target formula $\varphi(x;\bar y)$, we can choose the
  ambient quotient-counting partition given by
  \cref{lem:ambient-quotient-counting} so that every invariant profile of
  its parameter cells has a parameter-free pullback to the target class.
\end{enumerate}
Under these hypotheses, $\D$ is an $L_Q$-dimensional asymptotic class.
\end{corollary}

\begin{proof}
The proof of \cref{thm:quotient-transfer} uses quotient-trace reflection
only to descend the finitely many quotient-universe cells and the finitely
many invariant profiles attached to the chosen formula.  Hypotheses
\textup{(i)} and \textup{(ii)} provide exactly those target formulas.  Their conjunctions give the finite definable partition used in that proof.
The same counting and normalization arguments then apply.
\end{proof}

\subsection{The intrinsic denominator and interpretation dependence}

The number $L_Q$ is defined using the size of the target universe relative
to the chosen ambient class. It may therefore change when we change the
interpretation.  For example, fix
$a\geq2$ and consider
$
  \D_a=\big\{\F_{q^a}:q\text{ a prime power}\big\}
$
in the pure field language.  The identity interpretation of $\D_a$ in
itself has essential exponent $1$ and hence $L_Q=1$.  On the other hand,
after choosing an irreducible polynomial of degree $a$ over $\F_q$, the
usual construction on $\F_q^a$ uniformly parameter-interprets
$\F_{q^a}$ in $\F_q$.  Relative to the ambient size $q$, the interpreted
universe has size $q^a$, so this interpretation has essential exponent $a$
and $L_Q=a$.  This example shows dependence on the interpretation without imposing
trace reflection.

The proof of the transfer theorem also gives an intrinsic denominator.  Call a rational number $\alpha\in[0,1]$ a
\emph{realized target exponent} if there exist a target formula
$\varphi(x;\bar y)$ and a parameter-free definable cell containing target
structures of unbounded size such that, uniformly on that cell,
$$
  \abs{\varphi(A;\bar b)}
  =\mu\abs A^\alpha+o(\abs A^\alpha)
$$
for some $\mu>0$, and the fibers are unbounded when $\alpha>0$.  Write
$\operatorname{den}(\alpha)$ for the denominator of $\alpha$ in lowest
terms.

\begin{definition}
\label{def:visible-denominator}
Suppose that the interpretation satisfies quotient-trace reflection, or
the formula-wise profile descent hypotheses of
\cref{cor:profile-descent-criterion}.  Let $\Sigma_{\I}$ be the set of
fractions $d/e$ occurring on target-unbounded descended profile cells in
the proof of \cref{thm:quotient-transfer}, as the target formula varies.
Define
$$
  L_{\rm vis}(\I)
  =\lcm\left\{
    \frac{e}{\gcd(d,e)}:
    \frac de\in\Sigma_{\I}
  \right\},
$$
with the usual convention that the least common multiple of the empty set
is $1$.
\end{definition}

This denominator is independent of the interpretation.

\begin{proposition}
\label{prop:visible-least-denominator}
Assume quotient-trace reflection, or the hypotheses of
\cref{cor:profile-descent-criterion}.  Then
$
  L_{\rm vis}(\I)\mid L_Q,
$
and $L_{\rm vis}(\I)$ is the least asymptotic denominator of $\D$.
Hence, visible denominators computed from any two such
interpretations of the same target class are equal.
\end{proposition}

\begin{proof}
On a target-unbounded profile cell, the proof of
\cref{thm:quotient-transfer} gives an exponent $d/e$, where
$e\in\mathscr E_Q$ and $0\leq d\leq e$.  Its reduced denominator is
$e/\gcd(d,e)$, which divides $e$ and hence divides $L_Q$.  This proves
$L_{\rm vis}(\I)\mid L_Q$.

The descended profile cells, as the formula varies, form an
$L_Q$-dimensional presentation of $\D$.  Every positive exponent used in
that presentation has reduced denominator dividing $L_{\rm vis}(\I)$, so
the same estimates may be relabelled with denominator
$L_{\rm vis}(\I)$.  Thus $\D$ is
$L_{\rm vis}(\I)$-dimensional.

Conversely, suppose that $\D$ is $h$-dimensional.  Fix a fraction
$d/e\in\Sigma_{\I}$ and its target-definable profile cell.  Intersect this
cell with the finitely many cells in an $h$-dimensional partition for the
same formula.  At least one intersection supports structures of unbounded
size.  Uniqueness of a positive leading exponent on an unbounded sequence
forces
$
  \frac de=\frac jh
$
for one of the labels $j/h$ in that partition.  Hence
$e/\gcd(d,e)$ divides $h$.  Taking the least common multiple over
$\Sigma_{\I}$ shows that $L_{\rm vis}(\I)\mid h$.  Therefore
$L_{\rm vis}(\I)$ is the least possible denominator.
\end{proof}

Thus $L_Q$ is an upper bound depending on the interpretation, and
$L_{\rm vis}$ is the denominator determined by target-definable families.
The following criterion gives equality.

\begin{corollary}
\label{cor:coprime-denominator-witness}
Assume the hypotheses of \cref{prop:visible-least-denominator}.  Suppose
that for every $e\in\mathscr E_Q$ there is a realized descended profile
exponent $d/e$ with
$
  \gcd(d,e)=1.
$
Then $L_Q=L_{\rm vis}(\I)$,
so $L_Q$ is the least asymptotic denominator of $\D$.
\end{corollary}

\begin{proof}
The reduced denominator of the displayed exponent is $e$.  Hence every
$e\in\mathscr E_Q$ divides $L_{\rm vis}(\I)$, and therefore
$L_Q\mid L_{\rm vis}(\I)$.  The reverse divisibility is
\cref{prop:visible-least-denominator}.
\end{proof}

More generally, it is enough that the reduced denominators of the
witnesses account for every prime-power divisor of $L_Q$.  In the common single-exponent case
$\mathscr E_Q=\set{e_0}$, any target-definable family of exponent
$d/e_0$ with $\gcd(d,e_0)=1$ proves minimality.  The simplest useful
witness is an exponent $1/e_0$.  We use this in the polynomial-incidence examples. The whole structure has
leading size $q^k$, and the neighborhood of a curve has size $q$. Hence
the target exponent $1/k$ shows that the denominator $k$ is minimal.

\subsection{Intrinsic criteria for trace reflection}

We give a relative quantifier-elimination criterion for trace reflection.
Associate with $\I$ a class of two-sorted
\emph{interpretation pairs}.  For $A\in\D$, let
$M=\alpha(A)$ and let
$
  P_A=(M,A,\bar c_A,\pi_A),
$
where the two sorts carry their $L$- and $L'$-structures, the interpretation
parameter tuple $\bar c_A$ is named, and the graph of
$
  \pi_A:X_A\longrightarrow A
$
is the quotient map induced by the fixed isomorphism
$A\cong X_A/E_A$.  Equivalently, the $A$-sort may be viewed as the
corresponding imaginary sort.  Write
$
  \mathcal P_{\I}=\set{P_A:A\in\D}.
$

\begin{definition}
\label{def:separated-relative-qe}
We say that $\mathcal P_{\I}$ has \emph{separated relative quantifier
elimination over the target sort} if, for every parameter-free formula
$\eta(\bar y)$ of the pair language whose free variables lie in the
$A$-sort, there are parameter-free $L'$-formulas
$\chi_1(\bar y),\dots,\chi_t(\bar y)$ and pair-language sentences
$\sigma_1,\dots,\sigma_t$ such that, uniformly for $A\in\D$,
$$
  P_A\models\eta(\bar a)
  \quad\Longleftrightarrow\quad
  P_A\models
  \bigvee_{i=1}^t\bigl(\sigma_i\wedge\chi_i(\bar a)\bigr).
$$
We say that the interpretation pairs have \emph{sentence descent} if, for
every parameter-free pair-language sentence $\sigma$, there is a
parameter-free $L'$-sentence $\sigma^\downarrow$ such that
$$
  P_A\models\sigma
  \quad\Longleftrightarrow\quad
  A\models\sigma^\downarrow
$$
for every $A\in\D$.
\end{definition}

\begin{theorem}
\label{thm:relative-qe-trace}
If $\mathcal P_{\I}$ has separated relative quantifier elimination over the
target sort and sentence descent, then $\I$ is quotient-trace reflecting.
Hence, if $\C$ is an $N$-dimensional asymptotic class, then $\D$ is
an $L_Q$-dimensional asymptotic class.
\end{theorem}

\begin{proof}
Let $\rho(\bar v_1,\dots,\bar v_m;\bar z)$
be a parameter-free ambient formula which is invariant under $E_A$ in every
representative coordinate.  In the interpretation-pair language define
$$
\begin{aligned}
  \eta_\rho(y_1,\dots,y_m)\quad\Longleftrightarrow\quad
  \exists\bar v_1\cdots\exists\bar v_m\biggl(&
  \bigwedge_{j=1}^m\pi(\bar v_j)=y_j\\
  &{}\wedge\rho(\bar v_1,\dots,\bar v_m;\bar c_A)\biggr),
\end{aligned}
$$
where the quantifiers over the representative tuples are restricted to
$X_A$ and $\pi$ denotes the uniform quotient-map symbol, interpreted as
$\pi_A$ in $P_A$.  Invariance of $\rho$ gives, for arbitrary representatives
$\widetilde a_j$ of $a_j$,
$$
  P_A\models\eta_\rho(\bar a)
  \quad\Longleftrightarrow\quad
  M\models\rho(\widetilde a_1,\dots,\widetilde a_m;\bar c_A).
$$

Apply separated relative quantifier elimination to $\eta_\rho$.  Choose
$\chi_i(\bar y)$ and $\sigma_i$ such that
$$
  \eta_\rho(\bar y)
  \quad\Longleftrightarrow\quad
  \bigvee_{i=1}^t\bigl(\sigma_i\wedge\chi_i(\bar y)\bigr)
$$
uniformly on $\mathcal P_{\I}$.  By sentence descent, choose
$L'$-sentences $\sigma_i^\downarrow$ with the same truth values on the
corresponding target structures.  Then
$$
  \rho^\downarrow(\bar y)
  \defeq
  \bigvee_{i=1}^t
  \bigl(\sigma_i^\downarrow\wedge\chi_i(\bar y)\bigr)
$$
is a parameter-free $L'$-formula defining exactly the quotient trace of
$\rho$.  The case $m=0$ is the sentence-descent assumption itself.  Thus
$\I$ is quotient-trace reflecting, and the final assertion follows from
\cref{thm:quotient-transfer}.
\end{proof}

\begin{corollary}
\label{cor:parameter-free-relative-elimination}
Suppose that every parameter-free pair-language formula whose free variables
lie in the target sort is uniformly equivalent on $\mathcal P_{\I}$ to a
parameter-free $L'$-formula, including the case of no free variables.  Then
$\I$ is quotient-trace reflecting.
\end{corollary}

\begin{proof}
Take $t=1$ and the pair sentence to be the tautology in
\cref{def:separated-relative-qe}.  The empty-tuple case gives sentence
descent.  Apply \cref{thm:relative-qe-trace}.
\end{proof}

\begin{remark}
Uniform stable embeddedness alone is not enough. It gives a definition
$\psi(\bar y;\bar b_A)$ of a pair-definable relation using parameters from
$A$. Trace reflection requires a parameter-free target formula.  Every target-definable trace has
a canonical parameter in $A^{\rm eq}$, without any
elimination-of-imaginaries assumption.  We therefore need to show that the canonical parameter belongs to
$\operatorname{dcl}_{L'}^{\rm eq}(\varnothing)$ and that the resulting
parameter-free definition is uniform across pair completions. For sentences,
we need a separate condition on descent of completions.
\end{remark}

Let $T_{\I}=\Th(\mathcal P_{\I})$ be the common, possibly incomplete,
theory of the interpretation pairs, and let $T_{\rm tar}$ be its reduct to
the target language.  A \emph{completion} below always means a complete
theory extending the indicated common theory.

\begin{definition}
\label{def:target-controlled-stable}
We say that $T_{\I}$ has \emph{target-controlled stable embeddedness} if the
following conditions hold in every completion $T^*\supseteq T_{\I}$ and
every sufficiently saturated
$P=(M,A,\bar c,\pi)\models T^*$.
\begin{enumerate}[label=\textup{(\roman*)}]
  \item The target $A$, with its $L'$-structure, is sufficiently saturated,
  stably embedded in $P$, and the structure induced on it by $P$ is exactly
  its $L'$-structure.
  \item The restriction homomorphism
  $
    \Aut(P)\longrightarrow\Aut_{L'}(A)
  $
  is surjective.
  \item Pair completions are determined by target completions: if
  $T_1,T_2\supseteq T_{\I}$ are complete and have the same complete
  $L'$-reduct, then $T_1=T_2$.
\end{enumerate}
\end{definition}

\begin{theorem}
\label{thm:stable-embedded-descent}
If $T_{\I}$ has target-controlled stable embeddedness, then every
parameter-free pair-language formula whose free variables lie in the target
sort is, uniformly on $\mathcal P_{\I}$, equivalent to a parameter-free
$L'$-formula.  Hence $\I$ is quotient-trace reflecting.  If, moreover,
$\C$ is an $N$-dimensional asymptotic class, then $\D$ is an
$L_Q$-dimensional asymptotic class.
\end{theorem}

\begin{proof}
Let $\eta(\bar y)$ be a parameter-free formula in the pair language, with
free variables in the target sort. First, fix a completion
$T^*\supseteq T_{\I}$ and a sufficiently saturated
$P=(M,A,\bar c,\pi)\models T^*$.  By pure stable embeddedness, there is
an $L'$-formula $\psi(\bar y;\bar a)$ with a tuple $\bar a$ from $A$
such that
$$
  P\models\forall\bar y\,
  \bigl(\eta(\bar y)\leftrightarrow\psi(\bar y;\bar a)\bigr).
$$
Let
$$
  E_\psi(\bar u,\bar v)
  \quad\Longleftrightarrow\quad
  \forall\bar y\,
  \bigl(\psi(\bar y;\bar u)\leftrightarrow
        \psi(\bar y;\bar v)\bigr),
$$
and let $e=\bar a/E_\psi\in A^{\rm eq}$.  This is a canonical parameter of
the trace of $\eta$.

Let $g\in\Aut_{L'}(A)$.  By
\cref{def:target-controlled-stable}\textup{(ii)}, $g$ extends to an
automorphism $\widehat g$ of $P$.  Since $\eta$ is parameter-free,
$\widehat g$ preserves $\eta(A^{\abs{\bar y}})$, and hence $g(e)=e$ by
canonicity.  Saturation and homogeneity of the target give
$
  e\in\operatorname{dcl}_{L'}^{\rm eq}(\varnothing).
$
Choose a parameter-free $L'^{\rm eq}$-formula defining $e$ uniquely and
quantify over that unique imaginary in a definition of the trace.  A formula
in $L'^{\rm eq}$ with real free variables can be translated back to the
original language, so this gives a parameter-free $L'$-formula
$\chi_{T^*}(\bar y)$ such that
$$
  T^*\models\forall\bar y\,
  \bigl(\eta(\bar y)\leftrightarrow\chi_{T^*}(\bar y)\bigr).
$$

It remains to make this definition uniform over the completions of
$T_{\I}$.  By compactness, for each $T^*$ there is a pair-language sentence
$\theta_{T^*}\in T^*$ such that
$$
  T_{\I}\models\theta_{T^*}\longrightarrow
  \forall\bar y\,
  \bigl(\eta(\bar y)\leftrightarrow\chi_{T^*}(\bar y)\bigr).
$$
Finitely many such sentences $\theta_1,\dots,\theta_s$ cover all
completions of $T_{\I}$.

Condition \textup{(iii)} implies that each $\theta_i$ descends to a target
sentence.  Indeed, consider the restriction map from the Stone space of
completions of $T_{\I}$ to the Stone space of completions of
$T_{\rm tar}$.  By \textup{(iii)}, this map is a continuous injection and
hence a homeomorphism onto its closed image.  The clopen set defined by
$\theta_i$ therefore has clopen image in that closed subspace.  A clopen
subset of a closed subspace of a Stone space extends to a clopen subset of
the whole space, defined here by a target sentence $\tau_i$.  Thus
$T_{\I}\models\theta_i\leftrightarrow\tau_i.$

If $\chi_i$ is the target formula attached to $\theta_i$, then
$$
  \chi(\bar y)
  \defeq
  \bigvee_{i=1}^s\bigl(\tau_i\wedge\chi_i(\bar y)\bigr)
$$
is parameter-free and satisfies
$$
  T_{\I}\models\forall\bar y\,
  \bigl(\eta(\bar y)\leftrightarrow\chi(\bar y)\bigr).
$$
When $\bar y$ is empty, choose in each completion the appropriate truth
constant; the same Stone-space argument then gives sentence descent.
Therefore
\cref{cor:parameter-free-relative-elimination} gives quotient-trace
reflection, and the final assertion follows from
\cref{thm:quotient-transfer}.
\end{proof}

\begin{remark}
Stable embeddedness alone does not force empty-set descent.  If a point $a$
of a pure finite set $A$ is named in the pair, then $A$ is purely stably
embedded and the trace $x=a$ has canonical parameter $a$, but it is not
parameter-free definable in $A$.  The failure is precisely that target
automorphisms moving $a$ do not lift to the pair.

Likewise, let $P_n=(A_n,B_n)$ be a disjoint two-sorted structure, with
$A_n$ as target and with the theory of $B_n$ varying independently.  The
target can be purely stably embedded and satisfy automorphism lifting, while
a sentence about $B_n$ need not descend to $A_n$.  Here pair completions are
not determined by target completions.  Thus
\cref{def:target-controlled-stable}\textup{(ii)} and \textup{(iii)} address
independent obstructions.
\end{remark}

For a fixed formula, the proof of \cref{thm:stable-embedded-descent} only
uses part of the hypotheses in \cref{def:target-controlled-stable}. We now
state the conditions needed for each formula.  Let $\Gamma$ be a collection of parameter-free
pair-language formulas $\eta(\bar y)$ whose free variables lie in the target
sort.  For a completion $T^*\supseteq T_{\I}$ and a sufficiently saturated
$P=(M,A,\bar c,\pi)\models T^*$, whenever the trace
$\eta(A^{\abs{\bar y}})$ is target-definable, let $e_\eta$ denote a
canonical parameter of that trace in $A^{\rm eq}$.  Let
$$
  G_P=\operatorname{im}\bigl(
  \Aut(P)\longrightarrow\Aut_{L'}(A)
  \bigr).
$$

\begin{definition}
\label{def:local-target-control}
The pair theory $T_{\I}$ has \emph{$\Gamma$-local target control} if the
following conditions hold for every $\eta\in\Gamma$, every completion
$T^*\supseteq T_{\I}$, and every sufficiently saturated
$P=(M,A,\bar c,\pi)\models T^*$.
\begin{enumerate}[label=\textup{(\roman*)}]
  \item The trace $\eta(A^{\abs{\bar y}})$ is definable in the pure target
  structure.
  \item The pair-automorphism image controls the orbit of a canonical
  parameter of this trace: \,
  $
    \Aut_{L'}(A)\cdot e_\eta=G_P\cdot e_\eta.
  $
  \item Pair completions are $\eta$-determined by their target reducts:
  whenever complete $T_1,T_2\supseteq T_{\I}$ have the same complete
  target reduct, then, for every parameter-free target formula
  $\chi(\bar y)$,
  $$
  \begin{aligned}
    T_1\models\forall\bar y\,
      (\eta(\bar y)\leftrightarrow\chi(\bar y))
    \quad\Longleftrightarrow\quad
    T_2\models\forall\bar y\,
      (\eta(\bar y)\leftrightarrow\chi(\bar y)).
  \end{aligned}
  $$
\end{enumerate}
For formulas with no free variables, conditions \textup{(i)} and
\textup{(ii)} are vacuous, and \textup{(iii)} is read as a condition on
sentences.  The orbit condition is independent of the chosen canonical
parameter, since any two canonical parameters of the same definable set are
interdefinable.
\end{definition}

Condition \cref{def:local-target-control}\textup{(ii)} is strictly weaker
than surjectivity of the restriction homomorphism.  For each
$g\in\Aut_{L'}(A)$, it is enough to find $h\in\Aut(P)$ such that
$
  h|_A(e_\eta)=g(e_\eta);
$
the restriction of $h$ need not agree with $g$ away from this canonical
parameter.  Condition \textup{(iii)} is likewise weaker than determination
of the entire pair completion: it controls only which target formulas
define the trace of $\eta$.

\begin{theorem}
\label{thm:local-canonical-descent}
If $T_{\I}$ has $\Gamma$-local target control, then every
$\eta(\bar y)\in\Gamma$ is uniformly equivalent on
$\mathcal P_{\I}$ to a parameter-free $L'$-formula.
\end{theorem}

\begin{proof}
Fix $\eta\in\Gamma$ and a completion $T^*\supseteq T_{\I}$.  Work in a
sufficiently saturated $P=(M,A,\bar c,\pi)\models T^*$.  Suppose first
that $\bar y$ is nonempty.  By condition \textup{(i)}, the trace of $\eta$
has a canonical parameter $e_\eta\in A^{\rm eq}$.  Every automorphism of
$P$ fixes this parameter, because $\eta$ is parameter-free.  Hence
$G_P\cdot e_\eta=\set{e_\eta}$.  Condition \textup{(ii)} now gives
$
  \Aut_{L'}(A)\cdot e_\eta=\set{e_\eta},
$
and saturation and homogeneity give
$
  e_\eta\in\operatorname{dcl}_{L'}^{\rm eq}(\varnothing).
$
A parameter-free definition of $e_\eta$ in $A^{\rm eq}$, followed by the
standard translation of imaginary quantifiers to real target coordinates,
produces a parameter-free target formula $\chi_{T^*}(\bar y)$ such that
$$
  T^*\models\forall\bar y\,
  \bigl(\eta(\bar y)\leftrightarrow\chi_{T^*}(\bar y)\bigr).
$$
If $\bar y$ is empty, take $\chi_{T^*}$ to be the tautology or contradiction
according to the truth value of $\eta$ in $T^*$.

It remains to make the choice uniform over completions.  Let $X$ be the
Stone space of completions of $T_{\I}$, let $Y$ be the closed subspace of
target completions which occur as their reducts, and let $r:X\to Y$ be the
restriction map.  For a parameter-free target formula $\chi$, let
$$
  C_\chi=\left\{T^*\in X:
  T^*\models\forall\bar y\,
  (\eta(\bar y)\leftrightarrow\chi(\bar y))\right\}.
$$
Each $C_\chi$ is clopen.  By condition \textup{(iii)} it is saturated on
the fibers of $r$.  Since $X$ is compact and $Y$ is Hausdorff, the image
$r(C_\chi)$ is closed; its complement is the image of the saturated clopen
set $X\setminus C_\chi$, so $r(C_\chi)$ is also open.  It is therefore
defined, relative to $Y$, by a target sentence.

The sets $C_\chi$ cover $X$ by the first part of the proof.  Compactness
gives a finite subcover $C_{\chi_1},\dots,C_{\chi_s}$.  Choose target
sentences $\tau_i$ defining their images in $Y$ and disjointify these
sentences in the usual order.  Then
$
  \chi_\eta(\bar y)=
  \bigvee_{i=1}^s\bigl(\tau_i\wedge\chi_i(\bar y)\bigr)
$
is a parameter-free target formula uniformly equivalent to $\eta$ on all
interpretation pairs.
\end{proof}

For the asymptotic transfer, the family $\Gamma$ may be finite and chosen
separately for each target formula.

\begin{corollary}
\label{cor:local-semantic-profile-descent}
Assume that the fixed quotient-universe cells in
\cref{cor:profile-descent-criterion} have parameter-free target pullbacks.
For each target formula $\varphi(x;\bar y)$, choose one ambient
quotient-counting partition as in
\cref{lem:ambient-quotient-counting}, and let $\Gamma_\varphi$ be the
finite family of pair-language formulas defining its invariant parameter
profiles.  If $T_{\I}$ has $\Gamma_\varphi$-local target control for every
$\varphi$, then $\D$ is an $L_Q$-dimensional asymptotic class.
\end{corollary}

\begin{proof}
By \cref{thm:local-canonical-descent}, every invariant profile in the
chosen partition for $\varphi$ has a parameter-free target pullback.  The
two hypotheses of \cref{cor:profile-descent-criterion} therefore hold, and
that corollary gives the conclusion.
\end{proof}

Thus, for each chosen counting partition, we only need the following
conditions. The finitely many canonical parameters of the profile traces
must belong to $\operatorname{dcl}_{L'}^{\rm eq}(\varnothing)$, and the
finitely many sentences used to make these definitions uniform over
completions must descend to target sentences. We do not need lifting or
determination of completions for unrelated formulas.

\begin{corollary}
\label{cor:pure-quotient-dimension}
Let $\D$ be uniformly parameter-interpretable in an $N$-dimensional
asymptotic class.  Suppose that every target-unbounded cell in the fixed
quotient-universe partition has the same positive exponent $e_0/N$.  Then
$\D$ is a weak $e_0$-dimensional asymptotic class.  If the interpretation is
quotient-trace reflecting, or merely satisfies
\cref{cor:profile-descent-criterion}, then $\D$ is a full
$e_0$-dimensional asymptotic class.
\end{corollary}

\begin{proof}
In this case $\mathscr E_Q=\set{e_0}$, so $L_Q=e_0$.  Apply
\cref{cor:weak-quotient-transfer} for the weak conclusion and
\cref{thm:quotient-transfer,cor:profile-descent-criterion} for the full
conclusion.
\end{proof}

\begin{corollary}
\label{cor:canonical-quotient-trace}
Let $X_A/E_A$ be any uniformly definable quotient inside an
$N$-dimensional asymptotic class.  Expand the quotient by a relation symbol
for the descent of every parameter-free ambient formula which is invariant
under the relevant coordinatewise powers of $E_A$, allowing arity zero for
ambient sentences evaluated at the interpretation parameters.  The
resulting class of full quotient-trace structures is an asymptotic class,
with the dimension bound in \cref{thm:quotient-transfer}.
\end{corollary}

\begin{proof}
There are only countably many ambient formulas, and quotient-trace
reflection holds by the definition of the expanded language.  Apply
\cref{thm:quotient-transfer}.
\end{proof}

\begin{corollary}
\label{cor:selector-special}
Every selector-trace-reflecting interpretation is quotient-trace reflecting.
Hence, \cref{thm:quotient-transfer} recovers
\cref{thm:selector-transfer}, with the sharper essential denominator when
available.
\end{corollary}

\begin{proof}
Let $\rho$ be invariant under $E_A$.  Evaluate $\rho$ on the unique selected
representatives and apply selector-trace reflection.  Invariance shows that
the resulting target formula has the required equivalence for arbitrary
representatives.  The transfer statement follows from
\cref{thm:quotient-transfer}.
\end{proof}

\begin{theorem}
\label{thm:quotient-composition}
Let $\I$ uniformly parameter-interpret $\D$ in $\C$ with arity $r$, and let
$\mathcal J$ uniformly parameter-free interpret $\Eclass$ in $\D$ with
arity $q$.  If both interpretations are quotient-trace reflecting, then
their composite interpretation of $\Eclass$ in $\C$ has arity $rq$ and is
quotient-trace reflecting.
\end{theorem}

\begin{proof}
We use the usual composite interpretation. One element of an
$\Eclass$-structure is represented by a $q$-tuple of $\D$-elements, each
of which is represented by an $r$-tuple from the $\I$-domain. We translate
the domain, equivalence, and relation formulas of $\mathcal J$ through
$\I$.

Let $\rho$ be an ambient $L$-formula on tuples of composite representatives
and the composite interpretation parameters, and assume that $\rho$ is
invariant under the composite equivalence relation.  Replacing any one
$r$-block by an $\I$-equivalent block leaves the represented
$\D$-element unchanged and therefore gives a composite-equivalent
representative.  Hence $\rho$ is invariant in every $\I$-quotient
coordinate.  Since $\mathcal J$ is parameter-free, the only distinguished
parameters of the composite are those belonging to $\I$.  Quotient-trace
reflection for $\I$ therefore descends $\rho$ to a parameter-free
$L'$-formula on the corresponding $\mathcal J$-domain tuples.

Let $\bar u$ and $\bar v$ be equivalent domain tuples for $\mathcal J$,
and choose representatives of their coordinates under $\I$.  These
representatives are composite-equivalent.  Thus the original invariance of
$\rho$ makes the descended $L'$-formula invariant under the equivalence
relation of $\mathcal J$.  Quotient-trace reflection for $\mathcal J$ now
descends it to a parameter-free formula in the language of $\Eclass$.
Composing the two defining equivalences proves quotient-trace reflection for
the composite.
\end{proof}

\begin{corollary}
\label{cor:iterated-quotient-transfer}
In the setting of \cref{thm:quotient-composition}, if $\C$ is an asymptotic
class, then $\Eclass$ is an asymptotic class.  The same conclusion holds for
any finite chain of parameter-free quotient-trace-reflecting
interpretations.  More generally, the first interpretation in such a chain
may use parameters.
\end{corollary}

\begin{proof}
Apply \cref{thm:quotient-composition} and then
\cref{thm:quotient-transfer}; iterate for a longer finite chain.
\end{proof}

\begin{remark}
If $\mathcal J$ uses an interpretation-parameter tuple from a
$\D$-structure, then the entries of that tuple must be represented inside
the first interpretation.  An arbitrary ambient formula may distinguish
different such representatives, even though it is invariant in the object coordinates of the composite.  Quotient-trace reflection for $\I$ is
therefore not applicable without an additional canonical-representative or
selector hypothesis.  This problem does not occur when $\mathcal J$ is parameter-free.
\end{remark}

\begin{remark}
Theorem~\ref{thm:quotient-transfer} does not require definable choice.
The profiles are needed because the ambient asymptotic partition from
quotient counting need not be invariant under changing representatives,
although the quotient fiber cardinality is invariant.
\end{remark}

\section{Comparison with weak bi-interpretability}

We compare trace reflection with the weak bi-interpretability theorem of
Anscombe--Macpherson--Steinhorn--Wolf
\cite[Theorem~2.5.4]{AMSW2024}.

\begin{definition}
Let $A$ and $M$ be structures.  Suppose that $A$ is parameter-free
interpretable in $M$ by $\I$, and that $M$ is parameter-free interpretable
in $A$ by $\mathcal J$.  Let
$$
  f_A:A\longrightarrow A^{\I},
  \qquad
  g_M:M\longrightarrow M^{\mathcal J}
$$
be the interpretation isomorphisms.  Interpreting $g_M$ through $\I$
induces an isomorphism
$$
  g_M^{\I}:A^{\I}\longrightarrow A^{\I\mathcal J}.
$$
The canonical comparison map obtained by composing the two
interpretations is
$$
  h_A\defeq g_M^{\I}\circ f_A:
  A\longrightarrow A^{\I\mathcal J}.
$$
We say that $A$ is \emph{weakly bi-interpretable in $M$} if this canonical
map $h_A$ is parameter-free definable in $A$.
Uniform weak bi-interpretability of classes means that the two
interpretations and the formulas defining the maps $h_A$ are uniform.
No definability of the comparison map on the $M$-side is required.
\end{definition}

\begin{proposition}
\label{prop:weak-biimplies-trace}
Let $\I$ be the forward interpretation of a class $\D$ in a class $\C$ in a
uniform parameter-free weak bi-interpretation.  Then $\I$ is
quotient-trace reflecting.
\end{proposition}

\begin{proof}
Fix $A\in\D$, let $M=\alpha(A)$, and let
$\rho(\bar v_1,\dots,\bar v_m)$ be an ambient formula invariant under the
coordinatewise interpretation equivalence relation.  It therefore induces
a relation $R^{\I}$ on the interpreted copy $A^{\I}$ inside $M$.
Translate the formula defining $R^{\I}$ through the reverse interpretation
$\mathcal J$.  This produces, uniformly and without parameters, a definable
relation $R^{\I\mathcal J}$ on the double interpretation
$A^{\I\mathcal J}$ inside $A$.

Let $h_A:A\to A^{\I\mathcal J}$ be the definable comparison isomorphism
provided by weak bi-interpretability.  Define
$$
  R(\bar a)
  \quad\Longleftrightarrow\quad
  R^{\I\mathcal J}\bigl(h_A(\bar a)\bigr),
$$
where $h_A$ is applied coordinatewise.  This is a parameter-free formula in
the language of $A$.  By \cref{lem:translation} and the fact that $h_A$ is
the canonical round-trip isomorphism, $R(\bar a)$ holds precisely when
$\rho(\bar v_1,\dots,\bar v_m)$ holds for representatives of the images of
$\bar a$ under $\I$.  Invariance of $\rho$ makes the choice of
representatives irrelevant.  This is quotient-trace reflection.  The case
$m=0$ is the same argument for sentences.
\end{proof}

\begin{corollary}
In the classical finite-dimensional setting, the full transfer assertion
under parameter-free weak bi-interpretability is a special case of
\cref{thm:quotient-transfer}.
\end{corollary}

\begin{proof}
Combine \cref{prop:weak-biimplies-trace} with
\cref{thm:quotient-transfer}.
\end{proof}

\section{Strictness and model-theoretic consequences}

We give examples comparing trace reflection, bi-interpretability, and
weak bi-interpretability.

\begin{lemma}
\label{lem:fixed-finite-sort}
Let $\C_0$ be an $N$-dimensional asymptotic class, and let $B$ be a fixed
finite structure in a disjoint language.  Then the class
$
  \set{M\mathbin{\dot\cup}B:M\in\C_0}
$
in the disjoint two-sorted language is an $N$-dimensional asymptotic class.
Moreover, every ambient formula whose free variables lie in the $M$-sort is
uniformly equivalent to a formula in the language of $\C_0$.
\end{lemma}

\begin{proof}
Since the languages are disjoint, induction on formulas shows that every
formula is a finite Boolean combination of formulas in one sort. Writing
this combination in disjunctive normal form gives a finite disjunction of
conjunctions of an $M$-sort formula and a $B$-sort formula, with the free
variables separated by sort.  Only finitely many $B$-sort formulas occur in this separation, so only
finitely many truth patterns arise as the $B$-parameters vary.  Each pattern
is defined by the corresponding conjunction of those formulas and their
negations.  In particular, every resulting $B$-sentence
has a truth value independent of $M$.

If the object variable lies in the $M$-sort, then, after fixing the
finitely many possible truth values of the $B$-formulas, the fiber is defined
by one of finitely many $L(\C_0)$-formulas.  Apply the asymptotic-class axiom
to each of those formulas and refine by the finitely many $B$-conditions
which select it.  Writing $T=\abs M+\abs B$ for the cardinality of the
two-sorted structure, we have
$\abs M^{d/N}=T^{d/N}+o(T^{d/N})$, so the same labels and leading
coefficients give estimates relative to $T$.  If the object variable lies in
the $B$-sort, its fiber has one of the finitely many exact cardinalities
$0,1,\dots,\abs B$, and exact-count formulas give the required cells.
This proves the asymptotic-class assertion.  The same separation argument,
with all free variables in the $M$-sort, proves the final assertion.
\end{proof}

\begin{lemma}
\label{lem:aut-biinterpretation}
If two finite structures are parameter-free bi-interpretable, then their
automorphism groups are isomorphic.
\end{lemma}

\begin{proof}
A parameter-free interpretation sends every automorphism of the ambient
structure to an automorphism of the interpreted structure.  For a
bi-interpretation, the definable comparison isomorphisms between each
structure and its double interpretation show that the two induced
homomorphisms are inverse.  Hence the automorphism groups are isomorphic.
\end{proof}

\begin{example}
\label{ex:strictness}
Let
$
  \mathcal P=\set{\F_p:p\text{ prime}}
$
and let $U_2$ be a pure two-element structure in a disjoint sort.  Let
$
  \mathcal P^+=\set{\F_p\mathbin{\dot\cup}U_2:p\text{ prime}}.
$
The class $\mathcal P$ is a one-dimensional asymptotic class by
\cref{lem:subclass}, since finite fields form a one-dimensional asymptotic
class.  By \cref{lem:fixed-finite-sort}, so is
$\mathcal P^+$.

Interpret $\mathcal P$ in $\mathcal P^+$ by taking the field sort, with
equality as the interpretation equivalence relation.  The field sort itself
is a selector.  The final assertion of \cref{lem:fixed-finite-sort} gives
selector-trace reflection.  Thus the selector transfer theorem applies.

On the other hand, these two classes are not uniformly parameter-free
bi-interpretable.  For every prime $p$,
$
  \Aut(\F_p)=1
$ and  $\Aut(\F_p\mathbin{\dot\cup}U_2)\cong S_2.
$
This contradicts \cref{lem:aut-biinterpretation} for every possible pairing
of members of the two classes.  Thus trace reflection does not imply bi-interpretability.
\end{example}

\subsection{Strictness relative to weak bi-interpretability}

The preceding finite-sort example separates trace reflection from full
bi-interpretability.  The next example separates it from the weaker notion
used in \cref{prop:weak-biimplies-trace}.

Fix a prime $p$.  For $n\geq1$, let $U_n$ be a pure set of cardinality
$p^n$, let
$
  V_n=\F_p^n
$
in the vector-space language with addition, zero, and scalar multiplication
by every element of $\F_p$, and let
$
  M_n=(U_n,V_n)
$
be the disjoint two-sorted union.  Let
$$
  \mathcal U_p=\set{U_n:n\geq1},
  \qquad
  \mathcal M_p=\set{M_n:n\geq1}.
$$

\begin{lemma}
\label{lem:pure-vector-asymptotic}
The class $\mathcal M_p$ is a one-dimensional asymptotic class.
\end{lemma}

\begin{proof}
Let $\varphi(x;\bar y)$ be a formula, and let $\bar a$ be a parameter
tuple in $M_n$.
Suppose first that $x$ lies in the $U$-sort.  Let $C_{\bar a}$ be the finite
set of entries of $\bar a$ which lie in $U_n$.  The automorphism group of the
pure set $U_n$, fixing $C_{\bar a}$ pointwise and acting trivially on the
vector-space sort, is transitive on $U_n\setminus C_{\bar a}$.  Hence the
fiber $\varphi(U_n;\bar a)$ either contains all of
$U_n\setminus C_{\bar a}$ or is disjoint from it.  The two possibilities are
detected by the parameter formulas
$$
  \exists x\bigl(x\notin C_{\bar y}\wedge\varphi(x;\bar y)\bigr)
  \quad\text{and}\quad
  \exists x\bigl(x\notin C_{\bar y}\wedge\neg\varphi(x;\bar y)\bigr),
$$
where $x\in C_{\bar y}$ abbreviates the finite disjunction saying that $x$
is one of the $U$-sorted entries of $\bar y$.  These formulas cannot both
hold.  On the first cell the fiber has size $p^n+O(1)$; on the second it has
uniformly bounded size and is refined by the finitely many applicable
exact-count formulas.  If neither holds, then $U_n=C_{\bar a}$, so the whole
structure has bounded cardinality in terms of $\abs{\bar y}$ and the cell is
again split by exact fiber size.

Now suppose that $x$ lies in the $V$-sort.  Let $W_{\bar a}$ be the linear
span of the vector-sorted entries of $\bar a$.  Since $p$ and
$\abs{\bar y}$ are fixed, membership in $W_{\bar y}$ is expressed by the
finite disjunction
$
  x=\sum_i\lambda_i y_i,
  $\, \,$ \lambda_i\in\F_p,
$
over the vector-sorted coordinates.  The group of linear automorphisms of
$V_n$ fixing $W_{\bar a}$ pointwise, and acting trivially on $U_n$, is
transitive on $V_n\setminus W_{\bar a}$.  The same argument therefore shows
that the fiber is either co-bounded or bounded, with a bound depending only
on $p$ and $\abs{\bar y}$.  Refine every bounded-fiber cell by the finitely
many applicable exact-count formulas.  The exceptional case
$V_n=W_{\bar a}$ again supports only bounded structures and is handled by
the same refinement.

Finally,
$
  \abs{M_n}=2p^n.
$
A co-bounded fiber in either sort has size
$$
  p^n+O(1)=\frac12\abs{M_n}+o(\abs{M_n}),
$$
while all other unbounded-structure cells have bounded exact fibers.  The
cells just described, including their exact-count refinements, are
parameter-free definable, so the one-dimensional asymptotic class axiom
holds.
\end{proof}

\begin{lemma}
\label{lem:pure-sort-trace}
The parameter-free interpretation of $\mathcal U_p$ in $\mathcal M_p$ by the
$U$-sort is selector-trace reflecting.
\end{lemma}

\begin{proof}
The $U$-sort is its own selector.  Fix an ambient formula
$\rho(u_1,\dots,u_m)$ with all free variables in the $U$-sort, and let $q$
be its quantifier rank.  For a fixed equality type $\tau$ of $m$-tuples,
consider two expansions $(M_n,\bar u)$ and $(M_{n'},\bar u')$ in which the
distinguished tuples realize $\tau$.

A standard $q$-round back-and-forth argument shows that these expansions are
$q$-equivalent once $n$ and $n'$ are sufficiently large.  On the pure-set
sort, the duplicator preserves equality and always has a fresh element because
the two sets have arbitrarily many elements outside the distinguished tuple and
the previously played elements.  On the vector
sort, the duplicator maintains a linear isomorphism between the spans of the
played vectors.  A new vector in the existing span is matched using the
same linear combination; a vector outside the span is matched by a new
independent vector, which is available while the two ambient dimensions
exceed the number of rounds.  Since the sorts are disjoint, the two
strategies combine.  Hence, for every equality type $\tau$, the
truth value of $\rho$ on tuples of type $\tau$ is eventually constant as
$n\to\infty$.

There are only finitely many equality types.  Choose $n_0$ after which all
of their truth values have stabilized.  For each $n<n_0$, the pure-set
sentence asserting $\abs U=p^n$ is first-order expressible, and every
$S_{p^n}$-invariant relation on $U_n^m$ is a union of equality types.  Thus
we obtain a single pure equality formula by using the eventual union
of equality types at all other cardinalities and inserting the appropriate
union at each exceptional cardinality.  This formula has
exactly the same trace as $\rho$ in every $M_n$.  The case $m=0$ gives the
same conclusion for ambient sentences.
\end{proof}

\begin{lemma}
\label{lem:no-vector-in-pure-set}
There is no uniform parameter-free interpretation of $\mathcal M_p$ in
$\mathcal U_p$.
\end{lemma}

\begin{proof}
Suppose, towards a contradiction, that such an interpretation exists.
We allow an arbitrary reindexing of the two classes.  Along an unbounded set of
integers $n$, suppose that $M_n$ is interpreted in a pure set $U_{t(n)}$.
Let the sort representing $V_n$ be a quotient $X_n/E_n$ of a
parameter-free definable subset of $U_{t(n)}^r$, where the arity $r$ is
independent of $n$.  Let
$$
  m_n=\abs{U_{t(n)}}=p^{t(n)}.
$$
Since $\abs{X_n/E_n}=\abs{V_n}=p^n$ and
$\abs{X_n}\leq m_n^r$, we have $n\leq rt(n)$.

The symmetric group $S_{m_n}=\operatorname{Sym}(U_{t(n)})$ acts on
$X_n/E_n$.  Since the interpreted operations make this quotient isomorphic
to $V_n$, the action induces a homomorphism
$$
  h_n:S_{m_n}\longrightarrow\operatorname{Aut}(V_n)
  =\operatorname{GL}(n,p).
$$
The action of $S_{m_n}$ on $U_{t(n)}^r$ has at most $B_r$ orbits, where
$B_r$ is the $r$th Bell number: the orbit of a tuple is determined by its
equality pattern.  Hence $X_n$, and therefore its equivariant quotient
$X_n/E_n$, has at most $B_r$ orbits under $S_{m_n}$.

For all sufficiently large $n$,
$$
  \abs{\operatorname{GL}(n,p)}<p^{n^2}<m_n!.
$$
Indeed, $t(n)\geq n/r$, so $m_n\geq p^{n/r}$, while
$m_n!\geq(m_n/2)^{m_n/2}$ grows faster than $p^{n^2}$.  Thus $h_n$ is not
injective.  For $m_n\geq5$, the only normal subgroups of $S_{m_n}$ are
$1$, $A_{m_n}$, and $S_{m_n}$.  It follows that the image of $h_n$ has order
at most $2$.  Every $S_{m_n}$-orbit on the interpreted vector sort
therefore has size at most $2$, so that sort has at least $p^n/2$ orbits.
This contradicts the uniform bound $B_r$ for large $n$.
\end{proof}

\begin{theorem}
\label{thm:strict-weak-bi}
The classes $\mathcal U_p$ and $\mathcal M_p$ are one-dimensional
asymptotic classes.  The interpretation of $\mathcal U_p$ in
$\mathcal M_p$ is selector-trace reflecting, but $\mathcal U_p$ is not
uniformly weakly bi-interpretable in $\mathcal M_p$.
\end{theorem}

\begin{proof}
The class $\mathcal M_p$ is one-dimensional by
\cref{lem:pure-vector-asymptotic}.  The interpretation of $\mathcal U_p$
in it is selector-trace reflecting by \cref{lem:pure-sort-trace}; its
selector has arity one, so \cref{thm:selector-transfer} shows that
$\mathcal U_p$ is also one-dimensional.  A weak bi-interpretation would in
particular provide a uniform parameter-free reverse interpretation of
$\mathcal M_p$ in $\mathcal U_p$, contrary to
\cref{lem:no-vector-in-pure-set}.
\end{proof}

\begin{example}
\label{ex:matching-incidence-solution}
Fix the prime $p$ used above.  For $n\geq1$, let $H_n$ be the two-sorted
pure incidence structure
$
  H_n=(P_n,L_n,I_n),
$
where $\abs{P_n}=\abs{L_n}=p^n$ and $I_n$ is the graph of a bijection
from $P_n$ to $L_n$.  Thus $H_n$ is a perfect matching, viewed as an
incidence structure.  Let
$
  \mathcal H_p=\set{H_n:n\geq1}
$
and form the disjoint three-sorted expansion
$$
  \widehat M_n=(H_n,V_n),
  \qquad
  \widehat{\mathcal M}_p=\set{\widehat M_n:n\geq1},
$$
where $V_n=\F_p^n$ is in the vector-space language and there are no
relations between $V_n$ and the two incidence sorts.

Then $\mathcal H_p$ is a full one-dimensional asymptotic class, and every
member is a pure incidence structure omitting $K_{a,b}$ for all
$a,b\geq2$.  Moreover, it is obtained from the one-dimensional asymptotic
class
$\widehat{\mathcal M}_p$ by a selector-trace-reflecting interpretation,
whereas the two classes are not uniformly weakly bi-interpretable.
Hence, trace-reflection transfer applies to this interpretation even
though it is not part of a uniform weak bi-interpretation between the two
classes.
\end{example}

\begin{proof}
We first show that the ambient class is asymptotic. Given a parameter
tuple in $\widehat M_n$, we close its incidence entries under the matching
bijection.
This adds only boundedly many elements.  The automorphism group fixing that
finite closure pointwise is transitive on the remaining elements of
$P_n$, and also on the remaining elements of $L_n$.  On the vector sort,
the subgroup fixing the span of the vector parameters pointwise is
transitive outside that span.  Exactly as in the proof of
\cref{lem:pure-vector-asymptotic}, every one-variable definable fiber is
therefore either uniformly bounded or co-bounded in one of the three
sorts.  Bounded fibers are separated by exact-count formulas, while
$
  \abs{\widehat M_n}=3p^n
$
and every co-bounded fiber has size
$$
  p^n+O(1)=\frac13\abs{\widehat M_n}
  +o\bigl(\abs{\widehat M_n}\bigr).
$$
The cells are parameter-free definable by the same finite-closure and
finite-span formulas used there.  Hence
$\widehat{\mathcal M}_p$ is a one-dimensional asymptotic class.

Interpret $H_n$ in $\widehat M_n$ by taking the $P_n$- and $L_n$-sorts,
with equality as the interpretation equivalence relation.  The interpreted
domain is its own selector.  To prove selector-trace reflection, apply the
standard separation induction for disjoint many-sorted languages.  An
ambient formula all of whose free variables are in the incidence sorts is
a finite Boolean combination of incidence formulas and vector-space
sentences.  Each of the finitely many vector-space sentences involved has
an eventually constant truth value as $n\to\infty$: this follows, for
example, from the usual finite-round back-and-forth argument for vector
spaces of increasing dimension.  Each exceptional value of $n$ is
recognized inside $H_n$ by the first-order
sentence $\abs{P_n}=p^n$.  Replacing the vector-space sentences by their
eventual truth values and making these finitely many cardinality corrections
gives a uniform incidence formula with the same trace.  Thus the
interpretation is
selector-trace reflecting.  The selector transfer theorem now shows that
$\mathcal H_p$ is a full one-dimensional asymptotic class.

Every vertex of $H_n$ has incidence degree one.  It follows immediately
that $H_n$ contains no $K_{a,b}$ whenever $a,b\geq2$.

The two classes are not uniformly weakly bi-interpretable. The class
$\mathcal H_p$ is uniformly parameter-free interpretable in
$\mathcal U_p$: use two interpreted copies of $U_n$, and declare a point
and a line incident exactly when their underlying elements are equal.
If $\mathcal H_p$ and $\widehat{\mathcal M}_p$ were uniformly weakly
bi-interpretable, the reverse interpretation, composed with this
interpretation of $\mathcal H_p$ in $\mathcal U_p$, would uniformly
interpret $\widehat{\mathcal M}_p$ in $\mathcal U_p$.  Taking within
$\widehat M_n$ the $P_n$-sort as a pure set and retaining the vector sort
would then give a uniform interpretation of $\mathcal M_p$ in
$\mathcal U_p$.  This contradicts
\cref{lem:no-vector-in-pure-set}.  Hence no such weak bi-interpretation
exists.
\end{proof}

\begin{theorem}
\label{thm:complete-graph-incidence}
Fix a prime $p$.  For $n\geq1$, let
$
  B_n=(P_n,L_n,I_n)
$
be the pure vertex--edge incidence structure of the complete graph on the
pure set $U_n$ of size $p^n$; thus
$$
  P_n=U_n,\qquad L_n=[U_n]^2,
  \qquad I_n(u,\ell)\Longleftrightarrow u\in\ell.
$$
Let
$
  \mathcal B_p=\set{B_n:n\geq1}.
$
Then $\mathcal B_p$ is a full two-dimensional asymptotic class, and the
denominator $2$ is minimal.  Every incidence graph $B_n$ is connected,
every point has valency $p^n-1$, and $B_n$ is
$K_{a,b}$-free for all $a,b\geq2$.

Moreover, $\mathcal B_p$ is obtained from the one-dimensional asymptotic
class $\mathcal M_p$ by a quotient-trace-reflecting interpretation, but
$\mathcal B_p$ and $\mathcal M_p$ are not uniformly weakly
bi-interpretable.  Thus the quotient interpretation cannot be replaced by
a weak bi-interpretation between these two classes.
\end{theorem}

\begin{proof}
Recall that $M_n=(U_n,V_n)$, where $U_n$ is a pure set of size $p^n$ and
$V_n=\F_p^n$, and that $\mathcal M_p$ is one-dimensional by
\cref{lem:pure-vector-asymptotic}.  Interpret $B_n$ in $M_n$ with one
uniform domain of arity two.  A diagonal pair $(u,u)$ represents the point
$u$, while an off-diagonal pair $(u,v)$ represents the line $\{u,v\}$.
On $U_n^2$ define
\begin{align*}
  (u,v)E(u',v')\quad\Longleftrightarrow\quad{}
  &\bigl(u=v\wedge u'=v'\wedge u=u'\bigr)\\
  &{}\vee\Bigl(u\neq v\wedge u'\neq v'\\
  &\hspace{4em}{}\wedge
  \bigl((u=u'\wedge v=v')\vee(u=v'\wedge v=u')\bigr)\Bigr).
\end{align*}
The point and line predicates are respectively the images of the diagonal
and off-diagonal pairs, and incidence is defined by
$$
  I([(u,u)],[(v,w)])
  \quad\Longleftrightarrow\quad u=v\ \vee\ u=w.
$$
This gives a uniform parameter-free quotient interpretation of
$\mathcal B_p$ in $\mathcal M_p$.

We verify quotient-trace reflection.  Let
$\rho(\bar u_1,\ldots,\bar u_m)$ be an ambient formula on representative
pairs which is invariant under $E$ in every coordinate.  All its free
variables lie in the $U$-sort.  By \cref{lem:pure-sort-trace}, there is a
pure equality formula $\rho_0$ with exactly the same trace on every
$U_n^{2m}$.  The class $\mathcal B_p$ is uniformly parameter-free weakly
bi-interpretable with $\mathcal U_p$: in one direction use the quotient
interpretation just described without the vector sort, and in the other
direction take the point sort.  On returning to $B_n$, a line is sent to
the unordered pair of its two incident points, giving the required
definable round-trip comparison.  Therefore
\cref{prop:weak-biimplies-trace}, applied to the $E$-invariant formula
$\rho_0$, gives a parameter-free incidence formula defining its quotient
trace.  This is also the trace of $\rho$, proving quotient-trace reflection.

The $E$-classes have exact size one on the diagonal and exact size two off
the diagonal.  Writing $T=\abs{M_n}=2p^n$, the interpreted universe has
size
$$
  \abs{B_n}=p^n+\binom{p^n}{2}
  =\frac18T^2+o(T^2).
$$
Hence its only essential quotient-universe exponent is $2$, so
$L_Q=2$.  The quotient-trace transfer theorem now shows that
$\mathcal B_p$ is a full two-dimensional asymptotic class.  This
denominator is minimal: the parameter-free point predicate has size
$$
  \abs{P_n}=p^n
  =\sqrt{2}\,\abs{B_n}^{1/2}
   +o\bigl(\abs{B_n}^{1/2}\bigr),
$$
which is unbounded but sublinear in $\abs{B_n}$ and hence is incompatible
with the alternatives allowed in a one-dimensional asymptotic class.

The incidence graph is connected: two points are joined through their unique
common line, and every line is incident with its two endpoints.  Each point
lies on exactly $p^n-1$ lines.  Finally, two distinct points have exactly one
common incident line, and two distinct lines have at most one common incident
point.  Hence $B_n$ contains no $K_{2,2}$ and therefore no $K_{a,b}$ for any
$a,b\geq2$.

Suppose, towards a contradiction, that $\mathcal B_p$ and $\mathcal M_p$
were uniformly weakly bi-interpretable.  The reverse interpretation would
uniformly interpret $\mathcal M_p$ in $\mathcal B_p$.  Composing it with
the interpretation of $\mathcal B_p$ in the pure sets $\mathcal U_p$
would uniformly interpret $\mathcal M_p$ in $\mathcal U_p$, contrary to
\cref{lem:no-vector-in-pure-set}.  Thus no such weak bi-interpretation
exists.
\end{proof}

\begin{corollary}
\label{cor:ultraproduct}
Under the hypotheses of either \cref{thm:selector-transfer} or
\cref{thm:quotient-transfer}, every infinite ultraproduct of members of
$\D$ is supersimple of finite $\SU$-rank.  More precisely, the $\SU$-rank
is at most the transferred asymptotic denominator.
\end{corollary}

\begin{proof}
Apply the relevant transfer theorem and then
\cite[Corollary~2.8]{Elwes2007}, which states that an ultraproduct of an
$L$-dimensional asymptotic class is supersimple of $\SU$-rank at most $L$.
\end{proof}

\begin{remark}
For complementary results connecting pseudofinite counting dimensions with
simplicity and forking, see \cite{GarciaMacphersonSteinhorn2015}.
\end{remark}

\begin{corollary}
\label{cor:nonsimple-obstruction}
If a sequence of structures from $\D$ has an infinite non-simple
ultraproduct, then $\D$ cannot satisfy the conclusion of either transfer
theorem; in particular, no interpretation into an asymptotic class can
satisfy the corresponding hypotheses.
\end{corollary}

\begin{proof}
Otherwise \cref{cor:ultraproduct} would make that ultraproduct supersimple,
hence simple, a contradiction.
\end{proof}

\begin{remark}
Conant and Kruckman study the theory $T_{m,n}$ of existentially closed
incidence structures omitting $K_{m,n}$, where $K_{m,n}$ has $m$ points and
$n$ lines.  For $m,n\geq2$, they prove that $T_{m,n}$ is $\mathrm{NSOP}_1$
but not simple \cite{ConantKruckman2019}.  Hence, a sequence of
finite structures whose non-principal ultraproduct is a model of $T_{m,n}$
cannot lie in an asymptotic class.  This observation alone does not prove
that the class of all finite $K_{m,n}$-free incidence structures is not
asymptotic: we must separately produce, or identify, such an ultraproduct
from that finite class.
\end{remark}

\section{Projective quotient-trace structures}

We apply the quotient theorem to projective spaces. No definable choice
of projective coordinates is needed. The pure incidence geometries in the
following examples are classical.

\begin{definition}
\label{def:projective-trace}
Fix $k\geq2$.  For a finite field $F$, let
$
  X_k(F)=F^k\setminus\set{\bar0}
$
and define
$$
  \bar u\sim\bar v
  \quad\Longleftrightarrow\quad
  \exists t\in F^\times\ (\bar v=t\bar u).
$$
Let $\mathbb P_{\mathrm{tr}}^{k-1}(F)$ be the quotient $X_k(F)/{\sim}$
expanded by the descent of every parameter-free ring formula on tuples of
nonzero vectors which is invariant under independent nonzero scalar
multiplication in each tuple coordinate.  The empty tuple is allowed, so
the expansion also records the descents of parameter-free ring sentences.
\end{definition}

\begin{theorem}
\label{thm:projective-trace}
For each fixed $k\geq2$, the class
$$
  \mathcal P_k^{\mathrm{tr}}
  \defeq
  \set{\mathbb P_{\mathrm{tr}}^{k-1}(\F_q):q\text{ a prime power}}
$$
is a $(k-1)$-dimensional asymptotic class.  Moreover, $k-1$ is its least
possible asymptotic denominator.
\end{theorem}

\begin{proof}
Finite fields form a one-dimensional asymptotic class
\cite{CDM1992,MacphersonSteinhorn2008}.  The scalar relation is a uniformly
definable equivalence relation on $X_k(F)$, and quotient-trace reflection
holds by the definition of the target language.  Thus
\cref{cor:canonical-quotient-trace} applies.

Writing $q=\abs F$, the quotient universe has the exact cardinality
$$
  \frac{q^k-1}{q-1}=1+q+\cdots+q^{k-1}
  =q^{k-1}+o(q^{k-1}).
$$
Hence the quotient-universe partition has the single positive exponent
$e_0=k-1$, and \cref{cor:pure-quotient-dimension} gives the asserted
$(k-1)$-dimensional bound.

For minimality, the case $k=2$ is immediate.  Suppose $k\geq3$.  The
invariant ambient formula $u_1=0$ descends to a unary relation $H$ in the
full trace language, and
$$
  \abs{H(\mathbb P_{\mathrm{tr}}^{k-1}(\F_q))}
  =1+q+\cdots+q^{k-2}
  =q^{k-2}+o(q^{k-2}).
$$
Since the whole structure has size $q^{k-1}+o(q^{k-1})$, any
$N'$-dimensional asymptotic presentation would have, on an unbounded cell,
$
  \frac d{N'}=\frac{k-2}{k-1}
$
for some integer $d$.  The two integers $k-2$ and $k-1$ are coprime, so
$k-1$ divides $N'$.  Thus no smaller denominator is possible.
\end{proof}

\begin{corollary}
\label{cor:projective-incidence-trace}
For each finite field $F$, take a point copy and a line copy of
$(F^3\setminus\set{\bar0})/{\sim}$ and define
$$
  I([\bar x],[\bar a])
  \quad\Longleftrightarrow\quad
  a_1x_1+a_2x_2+a_3x_3=0.
$$
Expand the resulting two-sorted incidence structure by all quotient-trace
relations.  As $F$ ranges over the finite fields, these expansions form a
two-dimensional asymptotic class, and two is the least possible asymptotic
denominator.  Their incidence graphs are $K_{2,2}$-free and therefore
$K_{m,n}$-free for every $m,n\geq2$.
\end{corollary}

\begin{proof}
Code the two copies using a definable tag in $\set{0,1}$.  The quotient
universe has size
$$
  2(q^2+q+1)=2q^2+o(q^2),
$$
so \cref{cor:pure-quotient-dimension} gives dimension two.  The incidence
formula is invariant under independent scalar multiplication and hence is
one of the quotient-trace relations.  The unary quotient-trace relation
selecting point-copy classes with first homogeneous coordinate zero has
$q+1$ elements.  Relative to the total size $2q^2+o(q^2)$, its exponent is
$1/2$.  Thus any possible asymptotic denominator is even, and two is least.

If two distinct projective points lay on two distinct projective lines,
their two-dimensional span in $F^3$ would be annihilated by two linearly
independent covectors.  But the annihilator of a two-dimensional subspace
of $F^3$ is one-dimensional.  Thus no $K_{2,2}$ occurs.  Every
$K_{m,n}$ with $m,n\geq2$ contains a $K_{2,2}$.
\end{proof}

\begin{remark}
The full quotient-trace conclusion does not by itself imply asymptoticity of
the reduct to the pure incidence language, because the formulas defining
the asymptotic cells in the expansion need not be definable in the reduct.
For finite Moufang polygons, including the Desarguesian projective-plane
family above, pure-language asymptoticity is already known by work of Dello
Stritto \cite{DelloStritto2011}.
\end{remark}

\section{A polynomial incidence family}

In this section, we study a family of incidence structures with forbidden
bicliques. We first use field coordinates to prove weak asymptoticity and
quantifier-free counting results. We then reconstruct the coordinates from
a definable finite frame and use frame profiles to prove full asymptoticity
in the pure incidence language.

\begin{definition}
\label{def:polynomial-incidence}
Fix $k\geq2$. For a finite field $F$, let $P(F)\defeq F^2$ and
$C_k(F)\defeq F[X]_{<k}$. The coefficient map
$(a_0,\dots,a_{k-1})\mapsto\sum_{i=0}^{k-1}a_iX^i$ identifies $C_k(F)$
with $F^k$. Define an incidence relation between $P(F)$ and $C_k(F)$ by
$$
  I_k\bigl((x,y),f\bigr)
  \quad\Longleftrightarrow\quad
  y=f(x).
$$
Write $G_k(F)\defeq\bigl(P(F),C_k(F);I_k\bigr)$ for the resulting two-sorted
incidence structure, and set
$\Gclass_k\defeq\set{G_k(F):F\text{ a finite field}}$. We use the oriented
incidence convention: $K_{m,n}$ consists of $m$ point vertices and $n$ curve
vertices.
\end{definition}

\begin{lemma}\label{lem:polynomial-intersection}
Two distinct polynomials in $C_k(F)$ have at most $k-1$ common neighbors in
$P(F)$.
\end{lemma}

\begin{proof}
Let $f,g\in C_k(F)$ be distinct. A common neighbor $(x,y)$ satisfies
$f(x)=g(x)$, so $x$ is a root of the nonzero polynomial $f-g$, whose degree
is at most $k-1$. There are at most $k-1$ such values of $x$, and for each
$x$ the common value of $y$ is uniquely determined.
\end{proof}

\begin{corollary}
\label{cor:polynomial-biclique}
For every $n\geq2$ and every finite field $F$, the incidence structure
$G_k(F)$ is $K_{k,n}$-free.
\end{corollary}

\begin{proof}
A copy of $K_{k,n}$ would contain two distinct curve vertices with at least
$k$ common point neighbors, contradicting
\cref{lem:polynomial-intersection}.
\end{proof}

\begin{definition}
\label{def:coordinate-expansion}
Let $\widehat{G}_k(F)$ be the expansion of $G_k(F)$ obtained by adding:
\begin{enumerate}[label=\textup{(\roman*)}]
  \item a field sort carrying the ring structure of $F$;
  \item a coordinate map $P(F)\to F^2$; and
  \item the coefficient-coordinate map $C_k(F)\to F^k$.
\end{enumerate}
The coordinate maps are required to be bijections onto the indicated
Cartesian powers, and $I_k$ is related to them by the equation in
\cref{def:polynomial-incidence}. Let
$\widehat{\Gclass}_k\defeq
\set{\widehat{G}_k(F):F\text{ a finite field}}$.
\end{definition}

\begin{theorem}
\label{thm:coordinate-polynomial}
For every $k\geq2$, the class $\widehat{\Gclass}_k$ is a
$k$-dimensional asymptotic class.  Moreover, $k$ is its least possible
asymptotic denominator.
\end{theorem}

\begin{proof}
The class of finite fields is a one-dimensional asymptotic class by the
finite-field counting theorem of Chatzidakis, van den Dries, and Macintyre
\cite{CDM1992}.  Uniformly in $F$, the three sorts of $\widehat{G}_k(F)$ are the field
sort $F$, the point sort $P(F)=F^2$, and the curve sort
$C_k(F)=F[X]_{<k}\cong F^k$. Every relation and function of the coordinate
expansion is definable in the field.

For the one-sorted interpretation, we use three fixed definable tags,
for example $(0,0)$, $(0,1)$, and $(1,0)$ in $F^2$, and extend the coordinate
tuples with zeros.  This gives a parameter-free interpretation with unique
representatives in a fixed Cartesian power.  Conversely, because the target
contains the field sort with its ring structure and the coordinate maps,
every ambient field formula on selected representatives translates uniformly
back to a target formula.  Thus the interpretation is selector-trace
reflecting.

Writing $q=\abs F$, the selected universe has size
$s(F)=q+q^2+q^k$.
If $k=2$, then $s(F)=2q^2+q=2q^2+o(q^2)$; if $k>2$, then
$s(F)=q^k+o(q^k)$.  Hence every unbounded selector cell has the single
ambient exponent $k$.  By \cref{cor:pure-selector-dim},
$\widehat{\Gclass}_k$ is $k$-dimensional.

For minimality, the field sort is parameter-free definable and has size
$q$, while $s(F)=\lambda q^k+o(q^k)$ for a fixed $\lambda>0$.  If the class
were $N'$-dimensional, the defining formula for the field sort would have,
on an unbounded cell, an exponent $d/N'$ satisfying
$$
  \frac{d}{N'}=\lim_{q\to\infty}
  \frac{\log q}{\log s(F)}=\frac1k.
$$
Thus $k$ divides $N'$.  Since $k$ itself works, it is the least possible
denominator.
\end{proof}

\begin{proposition}
\label{prop:weak-pure-dimension}
Let $\C$ be an $N$-dimensional asymptotic class, and let $\D$ be uniformly
parameter-interpretable in $\C$ by an interpretation admitting a uniform
definable selector.  Suppose that, in one asymptotic partition for the
selector size, every target-unbounded cell has exponent $e_0/N$ for the
same positive integer $e_0$.  Then $\D$ is a weak $e_0$-dimensional asymptotic class.
No trace-reflection hypothesis is required.
\end{proposition}

\begin{proof}
Fix $\varphi(x;\bar y)$.  Use the selector to translate its fibers to
ambient definable sets with exactly the same cardinality, as in
\cref{cor:cardinality-selector}.  Apply tuple counting to the selector and
to the translated fiber, and pull the resulting joint cells back
set-theoretically to target parameter tuples by selected representatives.
These cells need not be definable in the target language, which is why the
conclusion is only weak asymptoticity.

Target-bounded cells are split by exact fiber cardinality.  On every
unbounded cell, the selector exponent is $e_0$, and
\cref{lem:selector-normalization} gives
$$
  \abs{\varphi(A;\bar b)}
  =\nu\abs A^{d/e_0}+o(\abs A^{d/e_0})
$$
for some $0\leq d\leq e_0$ and $\nu>0$, or gives the zero alternative.
Thus every exponent is of the form $d/e_0$, and the finitely many
set-theoretic cells form a weak $e_0$-dimensional asymptotic partition.
\end{proof}

\begin{corollary}
\label{cor:pure-polynomial-weak}
For every $k\geq2$, the pure two-sorted incidence class $\Gclass_k$ is a
weak $k$-dimensional asymptotic class, and every member is
$K_{k,n}$-free for every $n\geq2$.
\end{corollary}

\begin{proof}
The forbidden-biclique assertion is \cref{cor:polynomial-biclique}. Interpret
$G_k(F)$ directly in $F$ using the two definable sorts $P(F)=F^2$ and
$C_k(F)=F[X]_{<k}\cong F^k$, with fixed tags in a standard one-sorted coding.
Equality is the interpretation equivalence relation, so the coded domain
itself is a uniform selector. Writing $q=\abs F$, its size is $q^2+q^k$,
which has the single unbounded ambient exponent $k$ (with leading
coefficient $2$ when $k=2$ and $1$ when $k>2$).  Apply
\cref{prop:weak-pure-dimension} to the one-dimensional asymptotic class of
finite fields.
\end{proof}

\subsection{The linear case: biaffine planes}

For $k=2$, the pure incidence structure is the Desarguesian biaffine plane,
obtained from an affine plane by deleting one parallel class of lines.
We first prove full asymptoticity in the pure incidence language for this
case.

\begin{definition}
For a finite field $F$, let
$
  \Pi^{\mathrm{fl}}(F)
  =\bigl(\mathrm{PG}(2,F),p_\infty,\ell_\infty\bigr),
$
where $p_\infty$ is a point, $\ell_\infty$ is a line, and
$p_\infty I\ell_\infty$.  Let $\mathcal{PG}^{\mathrm{fl}}_2$ be the class of
all such flagged Desarguesian projective planes.
\end{definition}

\begin{proposition}
\label{prop:flagged-PG2}
The class $\mathcal{PG}^{\mathrm{fl}}_2$ is a two-dimensional asymptotic
class.
\end{proposition}

\begin{proof}
Dello Stritto proved that the pure incidence class
$\set{\mathrm{PG}(2,F):F\text{ finite}}$ is an asymptotic class
\cite{DelloStritto2011}.  This class is also uniformly parameter-free
interpretable in finite fields by taking both points and
lines to be the quotient
$
  (F^3\setminus\set{0})/F^\times,
$
with incidence given by the vanishing of the usual dot product between a
point vector and a line vector.  Its total universe has size
$ 2(q^2+q+1)=2q^2+O(q),
  $ and $ q=\abs F$.
Thus the quotient universe has the single positive ambient exponent $2$.
By \cref{cor:pure-quotient-dimension}, the pure projective-plane class is a
weak two-dimensional asymptotic class.  Combining this with Dello Stritto's
full asymptoticity and applying \cref{lem:denominator-refinement} shows that
it is in fact a full two-dimensional asymptotic class.

Finally, naming an incident point-line pair preserves the denominator by
\cref{lem:naming-parameters}, giving the assertion for
$\mathcal{PG}^{\mathrm{fl}}_2$.
\end{proof}

\begin{proposition}
\label{prop:biaffine-completion}
The class $\Gclass_2$ is uniformly parameter-free bi-interpretable with
$\mathcal{PG}^{\mathrm{fl}}_2$.
\end{proposition}

\begin{proof}
We describe the two interpretations and the comparison maps for their
composites.

Start with $G_2(F)=(P,L,I)$.  On the point sort define
$$
  E_{\mathrm v}(u,v)
  \quad\Longleftrightarrow\quad
  u=v\ \vee\
  \neg\exists\ell\bigl(I(u,\ell)\wedge I(v,\ell)\bigr).
$$
For $u=(x,y)$ and $v=(x',y')$, two distinct points lie on a common
nonvertical line exactly when $x\neq x'$.  Hence $E_{\mathrm v}$ is the
equivalence relation of having the same first coordinate.  Its quotient
$V=P/E_{\mathrm v}$ serves as the sort of vertical affine lines.

On the line sort define
$$
  E_\parallel(\ell,m)
  \quad\Longleftrightarrow\quad
  \ell=m\ \vee\
  \neg\exists u\bigl(I(u,\ell)\wedge I(u,m)\bigr).
$$
Two distinct lines $y=ax+b$ and $y=a'x+b'$ are disjoint exactly when
$a=a'$, so $E_\parallel$ is equality of slope.  Its quotient
$H=L/E_\parallel$ serves as the set of nonvertical ideal points.

Use two definable singleton quotient sorts, denoted
$\set{p_\infty}$ and $\set{\ell_\infty}$.  Interpret a projective point sort
and line sort by
$ P^*=P\mathbin{\dot\cup}H\mathbin{\dot\cup}\set{p_\infty},
 $ and $
  L^*=L\mathbin{\dot\cup}V\mathbin{\dot\cup}\set{\ell_\infty}.
$
Incidence extends the original relation as follows:
\begin{enumerate}[label=\textup{(\roman*)}]
  \item $u\in P$ is incident with $[v]_{E_{\mathrm v}}\in V$ precisely when
  $E_{\mathrm v}(u,v)$;
  \item $[\ell]_{E_\parallel}\in H$ is incident with $m\in L$ precisely
  when $E_\parallel(\ell,m)$;
  \item every element of $H$ and the point $p_\infty$ are incident with
  $\ell_\infty$;
  \item $p_\infty$ is incident with every element of $V$;
  \item there are no further incidences.
\end{enumerate}
All of these clauses are parameter-free definable on the relevant
quotients.

The resulting flagged plane is isomorphic to $\Pi^{\mathrm{fl}}(F)$.  An
explicit isomorphism is given by
\begin{align*}
  (x,y)&\longmapsto[x:y:1],\\
  [\ell_{a,b}]_{E_\parallel}&\longmapsto[1:a:0],\\
  p_\infty&\longmapsto[0:1:0],\\
  \ell_{a,b}&\longmapsto\set{aX-Y+bZ=0},\\
  [(c,d)]_{E_{\mathrm v}}&\longmapsto\set{X-cZ=0},\\
  \ell_\infty&\longmapsto\set{Z=0}.
\end{align*}
The displayed formulas also verify every incidence clause.

Conversely, start with a flagged projective plane
$(\Pi,p_\infty,\ell_\infty)$.  Define
$$
  P_0=\set{u:u\not I\ell_\infty},
  \qquad
  L_0=\set{\ell:\ell\neq\ell_\infty\text{ and }p_\infty\not I\ell},
$$
and retain the induced incidence relation.  After sending the distinguished
flag to the standard flag by a projective collineation, this reduct is
exactly $G_2(F)$.

The round trip from $G_2(F)$ to the completion and back is the identity on
$P$ and $L$.  For the other round trip, affine points and nonvertical lines
are again fixed.  An ideal point
$h\in\ell_\infty\setminus\set{p_\infty}$ is sent definably to the
$E_\parallel$-class of any line in $L_0$ incident with $h$; all such lines
belong to one class.  A line
$v\neq\ell_\infty$ through $p_\infty$ is sent to the
$E_{\mathrm v}$-class of any affine point incident with $v$; all such
points belong to one class.  The distinguished point and line go to the two
singleton sorts.  These descriptions give uniform parameter-free definable
comparison isomorphisms, proving bi-interpretability.
\end{proof}

\begin{theorem}
\label{thm:biaffine-asymptotic}
The pure incidence class $\Gclass_2$ is a two-dimensional asymptotic class,
and two is its least possible asymptotic denominator.  Every member is
$K_{2,n}$-free for every $n\geq2$.
\end{theorem}

\begin{proof}
By \cref{prop:flagged-PG2,prop:biaffine-completion} and the transfer of full
asymptoticity under parameter-free bi-interpretability
\cite[Theorem~2.5.4]{AMSW2024}, $\Gclass_2$ is an asymptotic class.  By
\cref{cor:pure-polynomial-weak}, it is also a weak two-dimensional
asymptotic class.  The denominator-refinement lemma now makes it a full
two-dimensional asymptotic class.

Its total cardinality is
$
  \abs{G_2(F)}=2q^2.
$
For every line $\ell$, the definable neighborhood
$\set{p:I(p,\ell)}$ has cardinality
$
  q=2^{-1/2}\abs{G_2(F)}^{1/2}.
$
Thus no one-dimensional asymptotic presentation is possible, and the least
denominator is two.  Finally, any two distinct points lie on at most one
line, so a $K_{2,n}$ with $n\geq2$ cannot occur.
\end{proof}

\subsection{Quantifier-free counting in all degrees}

We first prove counting results for quantifier-free formulas with one
object variable. We will then prove full asymptoticity in every degree.

\begin{lemma}
\label{lem:vertical-equivalence-general}
For every $k\geq2$, the formula
$$
  E_x(u,v)
  \quad\Longleftrightarrow\quad
  u=v\ \vee\
  \neg\exists c\bigl(I_k(u,c)\wedge I_k(v,c)\bigr)
$$
defines on $P(F)$ the equivalence relation of having the same first
coordinate.
\end{lemma}

\begin{proof}
If $u=(x,y)$ and $v=(x,y')$ are distinct, no polynomial graph contains both.
If their first coordinates are distinct, prescribing the two values gives
two independent linear conditions on the $k$ coefficients.  There are
$q^{k-2}$ solutions, in particular at least one, which proves the claim.
\end{proof}

\begin{lemma}
\label{lem:qf-point-counting}
Fix $k\geq2$ and a quantifier-free formula $\psi(u;\bar y)$ with $u$ in the
point sort. There is a finite parameter-free definable partition such that,
on every cell, with $q=\abs F$, the fiber has one of the forms
$q^2+O(q)$, $cq+O(1)$, or $O(1)$, where $c$ is a positive integer from a
finite set depending only on $\psi$. Every bounded cell may be refined to
exact cardinality.
\end{lemma}

\begin{proof}
After separating the finitely many cases in which $u$ equals a named point
parameter, the formula is a Boolean combination of sets
$
  N(\ell)=\set{u:I_k(u,\ell)}
$
for the finitely many curve parameters $\ell$ occurring in $\psi$.  Replace
repeated curve parameters by their equality classes.  Each $N(\ell)$ has
size $q$, while the intersection of two distinct such sets has size at most
$k-1$ by \cref{lem:polynomial-intersection}.

The Boolean algebra generated by $t$ distinct curve neighborhoods has one
atom outside their union, of size $q^2-tq+O(1)$; for each curve, the atom
lying on that curve and on no other has size $q+O(1)$; every remaining atom
is contained in the intersection of two distinct curves and is uniformly
bounded.  A Boolean combination is a disjoint union of a fixed selection of
these atoms.  It therefore has one of the three displayed forms, with $c$
the number of selected one-curve atoms when the outside atom is absent.

The equality pattern of the curve parameters is definable.  Every bounded
intersection has size bounded in terms of $k$ and the formula, and its size
can be separated by an exact-count formula.  The incidences of the
finitely many named point parameters with the relevant curves are also
first-order conditions.  These data give the required finite
parameter-free definable partition.
\end{proof}

\begin{lemma}
\label{lem:qf-curve-counting}
Fix $k\geq2$ and a quantifier-free formula $\psi(c;\bar y)$ with $c$ in the
curve sort. There is a finite parameter-free definable partition such that,
on every cell, with $q=\abs F$, the fiber has one of the forms
$c_0q^d+O(q^{d-1})$, where $1\leq d\leq k$, or $O(1)$. Here $c_0$ is a
positive integer, and the finitely many possible pairs $(d,c_0)$ depend only
on $\psi$. Every bounded cell may be refined to exact cardinality.
\end{lemma}

\begin{proof}
First separate the finitely many cases in which $c$ equals a named curve
parameter.  The remaining formula is a Boolean combination of the affine
hyperplanes
$$
  H_u=\set{c:I_k(u,c)}\subseteq C_k(F)
$$
attached to the finitely many point parameters $u$ occurring in $\psi$.
For a subset $S$ of those parameters, let
$
  A_S=\bigcap_{u\in S}H_u.
$
The set $A_S$ is empty exactly when the corresponding interpolation
conditions are inconsistent.  If it is nonempty and the points in $S$
occupy $r$ distinct $E_x$-classes, then
$$
  \dim A_S=
  \begin{cases}
    k-r,&r<k,\\
    0,&r\geq k.
  \end{cases}
$$
Indeed, evaluation at $r$ distinct first coordinates has Vandermonde rank
$\min(r,k)$; for $r\geq k$, a nonempty fiber consists of the unique
polynomial satisfying all the conditions.

Every atom of the Boolean algebra generated by the $H_u$ has the form
$
  A_S\setminus\bigcup_{v\notin S}H_v.
$
It is empty if some $H_v$ contains $A_S$.  Otherwise, when
$d=\dim A_S\geq1$, each nonempty intersection $A_S\cap H_v$ is a proper
affine hyperplane of $A_S$.  A union of finitely many proper affine
hyperplanes has $O(q^{d-1})$ points, so the atom has size
$
  q^d+O(q^{d-1}).
$
When $d=0$, the atom has size zero or one.  Since distinct atoms are
disjoint, a union of selected atoms has leading term $c_0q^d$, where $d$ is
the largest dimension among the selected nonempty atoms and $c_0$ is their
number.

The conditions on these intersections are uniformly definable in the pure
incidence language.  Nonemptiness of $A_S$ is expressed by
$
  \exists c\bigwedge_{u\in S}I_k(u,c),
$
containment $A_S\subseteq H_v$ by the corresponding universal implication,
and $r$ by the equality pattern modulo the definable relation $E_x$ from
\cref{lem:vertical-equivalence-general}.  There are only finitely many
subsets $S$, so these conditions give a finite parameter-free definable
partition.  Exact-count formulas handle the bounded cells.
\end{proof}

\begin{theorem}
\label{thm:qf-polynomial}
For every $k\geq2$, the class $\Gclass_k$ satisfies the
$k$-dimensional asymptotic-class axiom for every quantifier-free formula
$\varphi(x;\bar y)$ with one object variable.  More precisely, the cells are
parameter-free definable and every positive fiber has an asymptotic of the
form
$$
  \mu\abs{G_k(F)}^{d/k}
  +o\bigl(\abs{G_k(F)}^{d/k}\bigr),
  \qquad 0\leq d\leq k.
$$
\end{theorem}

\begin{proof}
Let $s=\abs{G_k(F)}=q^2+q^k$.  If $k>2$, then
$
  s=q^k(1+o(1)),
$ and $
  q^d=s^{d/k}(1+o(1))
$
for every fixed $0\leq d\leq k$.  If $k=2$, then $s=2q^2$ and
$
  q^d=2^{-d/2}s^{d/2}.
$
Apply \cref{lem:qf-point-counting} when the object variable is a point and
\cref{lem:qf-curve-counting} when it is a curve.  Their error terms are
little-$o$ of the corresponding positive main terms, and their parameter
partitions are finite and parameter-free definable.  Refining bounded cells
by exact cardinality gives the dimension-zero alternatives.
\end{proof}

\begin{corollary}
\label{cor:pure-denominator-rigid}
If $\Gclass_k$ is an asymptotic class with any finite denominator, then it is
a $k$-dimensional asymptotic class and $k$ is its least possible
denominator.
\end{corollary}

\begin{proof}
The class is weakly $k$-dimensional by
\cref{cor:pure-polynomial-weak}.  Hence any full finite-dimensional
asymptotic presentation may be refined to denominator $k$ by
\cref{lem:denominator-refinement}.  On the other hand, for a curve parameter
$c$ the neighborhood formula $I_k(x,c)$ has exactly $q$ point solutions.
Since $q=s^{1/k}(1+o(1))$ for $k>2$ and
$q=2^{-1/2}s^{1/2}$ for $k=2$, any denominator $N$ must realize the reduced
exponent $1/k$ as $d/N$.  Therefore $k\mid N$.
\end{proof}

\subsection{Framed profile descent}

The following theorem gives the descent argument needed for the polynomial
incidence structures. We assume that a finite frame defines coordinates
uniformly, without assuming a canonical choice of frame.

\begin{theorem}
\label{thm:framed-profile-transfer}
Let $\D$ be a class of finite $L'$-structures with finitely many sorts
$S_1,\dots,S_t$.  Suppose there are fixed integers
$n_1,\dots,n_t\geq1$, a fixed tuple of frame variables $\mathfrak f$, and a
parameter-free formula $\operatorname{Good}(\mathfrak f)$ with the following
properties.
\begin{enumerate}[label=\textup{(\roman*)}]
  \item Every $A\in\D$ has a good frame.
  \item For every good frame $\mathfrak f$ in $A$, uniformly in
  $(A,\mathfrak f)$, there is an interpreted finite field
  $K_{\mathfrak f}$ and definable bijections
  $
    h_{\mathfrak f,j}:S_j(A)\longrightarrow
    K_{\mathfrak f}^{n_j}
    \quad(1\leq j\leq t).
  $
  The field interpretation and the graphs of the bijections are given by
  fixed $L'$-formulas with $\mathfrak f$ in the displayed parameter
  positions.  The cardinality $q_A=\abs{K_{\mathfrak f}}$ is independent
  of the good frame.
  \item Under the union of these bijections, the $L'$-structure on $A$ is
  isomorphic to a fixed, uniformly ring-definable multisorted structure on
  the products $K_{\mathfrak f}^{n_j}$.  Equivalently, every $L'$-formula
  has a uniform ring-formula translation in these coordinates.
\end{enumerate}
Let
$
  e=\max_{1\leq j\leq t}n_j.
$
Then $\D$ is an $e$-dimensional asymptotic class.
\end{theorem}

\begin{proof}
Let $\varphi(x;\bar y)$ be an $L'$-formula with $x$ in $S_j$, and let
$
  \widehat\varphi(\bar X;\bar Y),
  \,\, \abs{\bar X}=n_j,
$
be its uniform ring-formula translation.  Since finite fields form a
one-dimensional asymptotic class, tuple counting gives a finite index set
$I$, parameter-free ring formulas $\theta_i(\bar Y)$ partitioning the
coordinate parameter tuples, and labels
$$
  (d_i,\mu_i)\in
  \bigl(\set{0,\dots,n_j}\times\R_{>0}\bigr)\cup\set{(0,0)}
$$
such that, on the $i$th cell in a field of size $q$,
$$
  \abs{\widehat\varphi(K^{n_j};\bar b)}
  =\mu_iq^{d_i}+o(q^{d_i})
$$
uniformly, with the usual meaning for the zero label.

For each $i\in I$, translate
$\theta_i(h_{\mathfrak f}(\bar y))$ back through the framed coordinate
system to an $L'$-formula $\Theta_i(\bar y;\mathfrak f)$.  For every
nonempty $P\subseteq I$, define the frame profile
$$
\begin{aligned}
  \Pi_P(\bar y)\quad\Longleftrightarrow\quad{}
  &\bigwedge_{i\in P}\exists\mathfrak f\,
    \bigl(\operatorname{Good}(\mathfrak f)
      \wedge\Theta_i(\bar y;\mathfrak f)\bigr)\\
  &{}\wedge
  \bigwedge_{i\notin P}\neg\exists\mathfrak f\,
    \bigl(\operatorname{Good}(\mathfrak f)
      \wedge\Theta_i(\bar y;\mathfrak f)\bigr).
\end{aligned}
$$
Every structure has a good frame, and each good frame places the coordinate
tuple in exactly one field cell. Hence the formulas $\Pi_P$ form a finite
parameter-free partition of the target parameter tuples.

Fix a profile on which the numbers $q_A$ are unbounded.  The target fiber
cardinality is independent of the chosen frame.  For each $i\in P$ and each
parameter tuple in the profile, choose a good frame witnessing $i$.  The
uniform finite-field estimate for the $i$th cell then holds for that same
target fiber.  Comparing the estimates along an unbounded sequence shows
that either all labels in $P$ are $(0,0)$ or all have the same positive
pair $(d_P,\mu_P)$.  Fixing one $i\in P$ and choosing a witnessing frame for
each tuple also shows that the resulting estimate is uniform on the whole
profile.  Thus the fiber is eventually empty, uniformly on the profile, in
the zero case; in the positive case,
$$
  \abs{\varphi(A;\bar b)}
  =\mu_Pq_A^{d_P}+o(q_A^{d_P}).
$$

If the numbers $q_A$ on a profile are bounded, then so are the target
structures, since
$
  \abs A=\sum_{j=1}^t q_A^{n_j}.
  $
Refine such a profile by exact fiber cardinality.  On an unbounded profile,
let
$\lambda=\abs{\set{j:n_j=e}}.
$
Then
$
  \abs A=\lambda q_A^e+O(q_A^{e-1}),
$
where the error is $O(1)$ if $e=1$.  Since $0\leq d_P\leq n_j\leq e$,
uniform normalization gives
$$
  \abs{\varphi(A;\bar b)}
  =\mu_P\lambda^{-d_P/e}\abs A^{d_P/e}
   +o\bigl(\abs A^{d_P/e}\bigr).
$$
The target-size uniformity follows from
$q_A^e\leq\abs A\leq t q_A^e$ for $q_A\geq1$.  Finally, merge the
finitely many profiles and exact-count refinements that
produce the same label.
This gives an $e$-dimensional asymptotic partition for $\varphi$.
\end{proof}

\subsection{Finite-direction completion and polynomial frames}

To prove full asymptoticity in the pure incidence language, we first
complete an affine plane uniformly after removing finitely many parallel
classes.

\begin{definition}
Fix $t\geq1$.  Let $F$ be a finite field and let
$T\subseteq\mathbb P^1(F)$ have cardinality $t$.  Write
$\mathcal A(F,T)$ for the two-sorted incidence structure whose point sort is
$F^2$ and whose line sort consists of precisely those affine lines whose
direction is not in $T$.
\end{definition}

\begin{theorem}
\label{thm:finite-direction-completion}
Fix $t\geq1$.  If $q=\abs F\geq t^2+3$, then the full affine plane
$\operatorname{AG}(2,F)$ is parameter-free interpretable in
$\mathcal A(F,T)$ by formulas uniform in the $t$-element set $T$.
The original point sort and the retained-line sort are carried identically
into the interpreted completion.
\end{theorem}

\begin{proof}
For distinct points $p,q$, let
$$
  R(p,q)
  \quad\Longleftrightarrow\quad
  p\neq q\ \wedge\
  \neg\exists\ell\bigl(I(p,\ell)\wedge I(q,\ell)\bigr).
$$
Thus $R(p,q)$ holds exactly when the direction of the full affine line
$pq$ lies in $T$.

For $R(p,q)$, define $\Lambda_t(p,q;r)$ by
$$
\begin{aligned}
  \Lambda_t(p,q;r)\quad\Longleftrightarrow\quad{}
  &r=p\ \vee\ r=q\\
  &{}\vee\exists z_1\dots z_{t^2}\Biggl(
    \bigwedge_{i<j}z_i\neq z_j\\
  &\hspace{8em}{}\wedge
    \bigwedge_{i=1}^{t^2}
    \bigl(R(z_i,p)\wedge R(z_i,q)\wedge R(z_i,r)\bigr)
  \Biggr).
\end{aligned}
$$
We claim that $\Lambda_t(p,q;r)$ says exactly that $r$ belongs to the
missing affine line through $p$ and $q$.

Let that line have direction $\delta\in T$.  If
$r\in pq\setminus\{p,q\}$, every point of $pq\setminus\{p,q,r\}$ is a
witness, so there are $q-3\geq t^2$ witnesses.

Suppose instead that $r\notin pq$, and let $z$ satisfy all three
$R$-conditions in the definition of $\Lambda_t$.  Let
$
  \alpha=\operatorname{dir}(pz)
 $ and $
  \beta=\operatorname{dir}(qz).
$
Both $\alpha$ and $\beta$ lie in $T$.  If $\alpha\neq\beta$, the two
corresponding lines through $p$ and $q$ meet in at most one point; hence
there are at most $t(t-1)$ possibilities for $z$.  If
$\alpha=\beta$, then necessarily
$\alpha=\beta=\delta$ and $z\in pq$.  For each direction
$\gamma\in T\setminus\{\delta\}$, the $\gamma$-line through $r$ meets
$pq$ in at most one point, so this case contributes at most $t-1$
possibilities.  Thus an off-line point has at most
$
  t(t-1)+(t-1)=t^2-1
$
witnesses, proving the claim.

Let
$
  X=\{(p,q):R(p,q)\}.
$
Define an equivalence relation on $X$ by
$$
  (p,q)\sim(p',q')
  \quad\Longleftrightarrow\quad
  \forall r\,
  \bigl(\Lambda_t(p,q;r)\leftrightarrow
        \Lambda_t(p',q';r)\bigr).
$$
The quotient $X/{\sim}$ is precisely the set of deleted affine lines.
Incidence between a point $r$ and the class of $(p,q)$ is defined by
$\Lambda_t(p,q;r)$.  Taking the disjoint union of this quotient with the
original retained-line sort gives all affine lines, with the original
incidence relation on the retained part.  This is a uniform
parameter-free interpretation of $\operatorname{AG}(2,F)$.
\end{proof}

\begin{lemma}
\label{lem:framed-affine-coordinates}
Let $\mathcal A=\operatorname{AG}(2,F)$, and let $O,U,V$ be noncollinear
points.  Uniformly over all finite fields $F$, the framed incidence
structure $(\mathcal A,O,U,V)$ interprets the field $F$ on the affine line
$OU$, with $O$ representing $0$ and $U$ representing $1$.  The following
relations are uniformly definable from the frame:
\begin{enumerate}[label=\textup{(\roman*)}]
  \item the two coordinates of every affine point;
  \item the direction of every affine line;
  \item the slope and intercept of every line not parallel to $OV$.
\end{enumerate}
The resulting coordinate maps are definable bijections with $F^2$, with
$\mathbb P^1(F)$, and with $F^2$, respectively.
\end{lemma}

\begin{proof}
Parallelism is definable in the affine incidence language: two lines are
parallel if they are equal or have no common point.  Let
$
  X=OU$ and
$  Y=OV.
$
The line $UV$ defines a bijection $\tau:X\to Y$: for $A\in X$,
$\tau(A)$ is the intersection with $Y$ of the line through $A$ parallel to
$UV$.  To verify the constructions, choose affine coordinates in which
$$
  O=(0,0),\qquad U=(1,0),\qquad V=(0,1).
$$
Then $\tau(a,0)=(0,a)$.

For $A=(a,0)$ and $B=(b,0)$ on $X$, let $B_Y=\tau(B)$.  Let $C$ be the
intersection of the line through $A$ parallel to $Y$ with the line through
$B_Y$ parallel to $X$.  The line through $C$ parallel to $UV$ meets $X$ at
$(a+b,0)$.  This gives a uniform incidence construction of addition.

For multiplication, let $B_Y=\tau(B)$ and take the line $UB_Y$.  The line
through $A$ parallel to $UB_Y$ meets $Y$ at $(0,ab)$; applying
$\tau^{-1}$ gives $(ab,0)$.  Hence the induced operations on $X$ are the
field operations of $F$.

The coordinates of an arbitrary point are obtained by projecting to $X$
and $Y$ along the two axis directions and then using $\tau^{-1}$ on the
$Y$-coordinate.  The line through $O$ parallel to a given line determines
its direction.  A nonvertical line has a unique slope and a unique
intersection with $Y$, giving its slope and intercept.  Every construction
uses a fixed finite incidence diagram and is therefore uniformly
first-order definable from $O,U,V$.
\end{proof}

We next define a partial affine plane in the polynomial incidence
structure.  Recall the definable relation $E_x$ from
\cref{lem:vertical-equivalence-general}.

\begin{lemma}
\label{lem:polynomial-slice}
Let $k\geq2$, let
$
  B=(p_1,\dots,p_{k-2})
$
be a tuple of points in pairwise distinct $E_x$-classes, and let
$
  S_B=\left\{c:\bigwedge_{i=1}^{k-2}I_k(p_i,c)\right\},
$
with $S_B=C_2(F)$ when $k=2$.  Let
$
  P_B=~\left\{u:\bigwedge_{i=1}^{k-2}\neg E_x(u,p_i)\right\}.
$
With point sort $S_B$, line sort $P_B$, and incidence inherited from
$I_k$, the resulting structure is an affine plane of order $q$ with
exactly $k-1$ parallel classes deleted.
\end{lemma}

\begin{proof}
Write $p_i=(a_i,b_i)$ and let $A=\{a_1,\dots,a_{k-2}\}$.  Choose
$c_0\in S_B$, with polynomial $f_0$, and set
$$
  g_A(X)=\prod_{a\in A}(X-a),
$$
where $g_A=1$ when $A$ is empty.  Every polynomial in $S_B$ has a unique
expression
$$
  f(X)=f_0(X)+g_A(X)(\alpha X+\beta),
  \qquad (\alpha,\beta)\in F^2.
$$
Thus $S_B$ is an affine plane with coordinate pair $(\alpha,\beta)$.

For $u=(x,y)\in P_B$, incidence with $f$ is equivalent to
$$
  \alpha x+\beta
  =\frac{y-f_0(x)}{g_A(x)}.
$$
As $y$ varies with $x$ fixed, these are all affine lines in one parallel
class.  The retained directions are indexed by $x\in F\setminus A$.
The omitted directions are the $k-2$ values in $A$, together with the
vertical direction.  Hence exactly $k-1$ parallel classes are absent.
\end{proof}

Fix $k\geq2$. We describe a finite parameter tuple which defines
polynomial coordinates in the incidence structure. Let
$q_k^0\defeq\max\set{(k-1)^2+3,k+1}$. We work over fields of size at least
$q_k^0$ and add the finitely many smaller structures at the end.

For $c,d\in S_B$, let
$
  \operatorname{Del}_B(c,d)
$
mean that $c\neq d$ and no retained line of the partial affine plane is
incident with both $c$ and $d$.  Equivalently, the unique line through $c$
and $d$ in the interpreted affine completion is one of the deleted lines.
This is a uniformly definable relation in the incidence language, with the
base tuple $B$ displayed among its parameters.

\begin{definition}
\label{def:polynomial-preframe}
A \emph{polynomial preframe} in $G_k(F)$ consists of the following data.
\begin{enumerate}[label=\textup{(\roman*)}]
  \item A curve $c_0$ and points
  $B=(p_1,\dots,p_{k-2})$ incident with $c_0$, in pairwise distinct
  $E_x$-classes.
  \item Three further points $r,s,t$ incident with $c_0$, whose
  $E_x$-classes are pairwise distinct and disjoint from the classes of
  the $p_i$.
  \item Curves $U,V\in S_B$ such that $U\neq c_0$, $I_k(r,U)$, and
  $\operatorname{Del}_B(c_0,V)$.  Thus $(c_0,U,V)$ is an affine frame in
  the completion, with $c_0U$ as the horizontal axis and $c_0V$ as the
  vertical direction.
  \item Curves $D_1,\dots,D_{k-2}\in S_B$ such that
  $\operatorname{Del}_B(c_0,D_i)$ for every $i$, and such that the deleted
  lines through $c_0$ represented by\\
  $
    c_0V, c_0D_1, \dots, c_0D_{k-2}
  $
  are pairwise distinct and exhaust all deleted directions.
  \item Nonzero elements
  $
    \lambda_0,\lambda_1,\dots,\lambda_{k-2}
  $
  of the field interpreted on the horizontal axis $c_0U$ by
  \cref{lem:framed-affine-coordinates}.
\end{enumerate}
When $k=2$, the tuples $B$, $(D_i)$, and $(\lambda_i)_{i\geq1}$ are empty.
All of these conditions are understood in the interpreted completion from
\cref{thm:finite-direction-completion}.
\end{definition}

Because $k$ is fixed, a single parameter-free formula in the language of
$G_k$ expresses the domain, quotient, exhaustiveness, and nondegeneracy
conditions in \cref{def:polynomial-preframe}.  We next describe the
coordinate map induced by a preframe.

Let $K_{\mathfrak f}$ be the field on $c_0U$.  For each $i$, let
$a_i\in K_{\mathfrak f}$ be the finite slope of the deleted line $c_0D_i$.
For $p\in P_B$, let $m(p)$ and $b(p)$ be the slope and intercept of the
retained line represented by $p$.  Let
$
  G_{\mathfrak f}(T)=\prod_{i=1}^{k-2}(T-a_i),
$
with $G_{\mathfrak f}=1$ when $k=2$.

For $E_x(u,p_i)$, let $c_i(u)$ be the unique curve incident with
$
  u,\, p_j\ (j\neq i),\, r,\, s.
$
These are exactly $k$ conditions at distinct first coordinates.  Let
$v_i(u)$ be the unique point on $c_i(u)$ satisfying $E_x(v_i(u),t)$, and
let
$
  h_i(u)=b\bigl(v_i(u)\bigr).
$
Define the point-coordinate relation
$
  \operatorname{Coord}_{\mathfrak f}(p;x,y)
$
by the following cases:
$$
  \begin{array}{ll}
  p\in P_B:
    &x=m(p),\qquad
     y=\lambda_0G_{\mathfrak f}(m(p))b(p),\\[2mm]
  E_x(p,p_i):
    &x=a_i,\qquad y=\lambda_i h_i(p).
  \end{array}
$$

\begin{lemma}
\label{lem:preframe-point-bijection}
For every polynomial preframe $\mathfrak f$, the relation
$\operatorname{Coord}_{\mathfrak f}$ is a uniformly definable bijection
$
  P(F)\longrightarrow K_{\mathfrak f}^2.
$
\end{lemma}

\begin{proof}
The slice retains every finite slope except
$a_1,\dots,a_{k-2}$, while the remaining deleted direction $c_0V$ is
vertical.  Hence, the first coordinate sends the $E_x$-classes
outside $B$ bijectively to
$
  K_{\mathfrak f}\setminus\{a_1,\dots,a_{k-2}\},
$
and sends the class of $p_i$ to $a_i$.

Fix a retained direction.  Lines of that slope are bijectively indexed by
their intercepts.  Since
$G_{\mathfrak f}(m)\neq0$ on a retained direction and $\lambda_0\neq0$,
the second-coordinate map is a bijection on each non-base $E_x$-class.

Fix $i$.  The map $u\mapsto v_i(u)$ is a bijection from the class of
$p_i$ to the class of $t$.  Indeed, interpolation gives existence and
uniqueness.  Alternatively, if $v_i(u)=v_i(u')$, the two curves
$c_i(u)$ and $c_i(u')$ agree at the $k$ distinct first-coordinate classes
represented by $p_j$ for $j\neq i$, together with $r$, $s$, and $t$,
so they are equal and hence $u=u'$.  Both classes have size $q$.
The intercept map on the fixed retained direction represented by $t$ is a
bijection onto $K_{\mathfrak f}$, and multiplication by
$\lambda_i\neq0$ preserves bijectivity.  The maps on these classes together give the required bijection.
\end{proof}

\begin{definition}
\label{def:good-polynomial-frame}
Let $\mathfrak f$ be a polynomial preframe.  For a curve $c$ and
$\bar e=(e_0,\dots,e_{k-1})\in K_{\mathfrak f}^k$, let
$\operatorname{Graph}_{\mathfrak f}(c;\bar e)$ assert
$$
  \forall p\,\forall x\,\forall y\,
  \left(
    \operatorname{Coord}_{\mathfrak f}(p;x,y)
    \longrightarrow
    \left[I_k(p,c)\leftrightarrow
    y=\sum_{j=0}^{k-1}e_jx^j\right]
  \right).
$$
The preframe is \emph{good} if
$$
  \forall c\,\exists!\bar e\,
  \operatorname{Graph}_{\mathfrak f}(c;\bar e)
$$
and
$$
  \forall\bar e\,\exists!c\,
  \operatorname{Graph}_{\mathfrak f}(c;\bar e).
$$
\end{definition}

All quantification over $K_{\mathfrak f}$ in
\cref{def:good-polynomial-frame} takes place in the uniformly interpreted
field and therefore translates into a first-order formula in the original
incidence language.  Hence goodness is a uniform parameter-free property of
the frame tuple.

\begin{lemma}
\label{lem:good-frame-existence}
For every fixed $k\geq2$ and every finite field $F$ with
$\abs F\geq q_k^0$, the structure $G_k(F)$ contains a good polynomial
frame.
\end{lemma}

\begin{proof}
We use the given field coordinates to construct a good frame. Choose a
polynomial $f_0$ of degree $<k$ and distinct elements
$
  a_1,\dots,a_{k-2},r,s,t\in F.
$
Let $A=\{a_1,\dots,a_{k-2}\}$, let $c_0=f_0$, and set
$p_i=(a_i,f_0(a_i))$.  We also use $r,s,t$ for the corresponding points on
the graph of $f_0$.  Set
$$
  g_A(X)=\prod_{i=1}^{k-2}(X-a_i).
$$
Choose
$$
\begin{aligned}
  U(X)&=f_0(X)+g_A(X)(X-r),\\
  V(X)&=f_0(X)+g_A(X),\\
  D_i(X)&=f_0(X)+g_A(X)(X-a_i).
\end{aligned}
$$
Every curve in $S_B$ has a unique representation
$$
  f_0(X)+g_A(X)\bigl(\alpha(X-r)+\beta\bigr).
$$
Relative to the affine frame $(c_0,U,V)$, its coordinates are
$(\alpha,\beta)$.  If $p=(x,y)$ and $x\notin A$, the line represented by
$p$ has equation
$$
  \beta=-(x-r)\alpha+\frac{y-f_0(x)}{g_A(x)}.
$$
Hence
$
  m(p)=r-x
  $ and $
  b(p)=\frac{y-f_0(x)}{g_A(x)}.
$
The direction $c_0V$ is vertical, while $c_0D_i$ has slope $r-a_i$.
Thus these curves provide the required exhaustive list of deleted
directions.

Let $a_i^\sharp=r-a_i$.  Then
$$
  \prod_i\bigl(m(p)-a_i^\sharp\bigr)
  =\prod_i(a_i-x)
  =(-1)^{k-2}g_A(x).
$$
Choose
$
  \lambda_0=(-1)^{k-2}.
$
For a base class, let
$$
  L_i(X)=
  \left(\prod_{j\neq i}\frac{X-a_j}{a_i-a_j}\right)
  \frac{X-r}{a_i-r}
  \frac{X-s}{a_i-s}.
$$
This polynomial has degree $k-1$, takes value $1$ at $a_i$, and vanishes
at every $a_j$ with $j\neq i$, as well as at $r$ and $s$.  If
$u=(a_i,y)$, then
$
  c_i(u)=f_0+(y-f_0(a_i))L_i.
$
Hence,
$$
  h_i(u)=\kappa_i\bigl(y-f_0(a_i)\bigr),
  \qquad
  \kappa_i=\frac{L_i(t)}{g_A(t)}\neq0.
$$
Choose $\lambda_i=\kappa_i^{-1}$.  With these choices,
$
  \operatorname{Coord}_{\mathfrak f}((x,y))
  =\bigl(r-x,\,y-f_0(x)\bigr)
$
for every point $(x,y)$.

If $c$ corresponds to a polynomial $f$, its coordinate graph is therefore
$$
  Y=f(r-X)-f_0(r-X),
$$
a polynomial in $X$ of degree $<k$.  Conversely, every polynomial of degree
$<k$ arises uniquely in this way.  Hence the frame is good.
\end{proof}

\begin{theorem}
\label{thm:framed-polynomial-reconstruction}
Fix $k\geq2$.  Let $F$ be a finite field with $\abs F\geq q_k^0$, and let
$\mathfrak f$ be a good polynomial frame in $G_k(F)$.  The completed slice
uniformly interprets a field $K_{\mathfrak f}$ of cardinality $\abs F$, and
there is a uniformly $\mathfrak f$-definable incidence isomorphism
$
  h_{\mathfrak f}:G_k(F)\longrightarrow G_k(K_{\mathfrak f}).
$
Conversely, $G_k(K)$ is uniformly parameter-free interpretable in every
finite field $K$.
\end{theorem}

\begin{proof}
Given $\mathfrak f$, the completed slice and
\cref{lem:framed-affine-coordinates} uniformly interpret the field
$K_{\mathfrak f}$ on the horizontal axis $c_0U$.  That axis has exactly
$\abs F$ points.  By \cref{lem:preframe-point-bijection,def:good-polynomial-frame}, the maps
$p\mapsto\operatorname{Coord}_{\mathfrak f}(p)$ and $c\mapsto\bar e$, where
$\operatorname{Graph}_{\mathfrak f}(c;\bar e)$, are uniformly
$\mathfrak f$-definable bijections from the point and curve sorts onto
$K_{\mathfrak f}^2$ and $K_{\mathfrak f}^k$, respectively.
By the definition of $\operatorname{Graph}_{\mathfrak f}$, these bijections
preserve incidence.  Their union is therefore the stated
isomorphism $h_{\mathfrak f}$.

Conversely, in a field $K$, take the point sort to be $P(K)=K^2$ and the
curve sort to be $C_k(K)=K[X]_{<k}$, identified with $K^k$ by coefficients,
and define incidence by $y=\sum_{j=0}^{k-1}a_jx^j$. This is a uniform
parameter-free interpretation of $G_k(K)$.
\end{proof}

\begin{theorem}
\label{thm:frame-profile-polynomial}
Fix $k\geq2$, and let $\Gclass_k^{\geq q_k^0}$ be the subclass of all
$G_k(F)$ with $F$ finite and $\abs F\geq q_k^0$. Then
$\Gclass_k^{\geq q_k^0}$ is a $k$-dimensional asymptotic class.
\end{theorem}

\begin{proof}
By \cref{lem:good-frame-existence}, every structure in the displayed class
has a good polynomial frame.  For every such frame,
\cref{thm:framed-polynomial-reconstruction} interprets a finite field
$K_{\mathfrak f}$ of cardinality $q=\abs F$ and gives uniformly definable
bijections from the point and curve sorts to
$K_{\mathfrak f}^2$ and $K_{\mathfrak f}^k$.  Under these bijections, the
pure incidence relation is given by the fixed
ring-definable evaluation relation
$
  y=\sum_{j=0}^{k-1}a_jx^j.
$
Thus the hypotheses of \cref{thm:framed-profile-transfer} hold with
sort exponents $2$ and $k$.  Their maximum is $k$, so
\cref{thm:framed-profile-transfer} gives the
conclusion.
\end{proof}

\begin{theorem}
\label{thm:full-polynomial-incidence}
For every $k\geq2$, the pure incidence class
$\Gclass_k=\set{G_k(F):F\text{ a finite field}}$ is a $k$-dimensional
asymptotic class, and $k$ is its least possible denominator. Moreover, every
member is $K_{k,n}$-free for every $n\geq2$.
\end{theorem}

\begin{proof}
The cofinal subclass $\Gclass_k^{\geq q_k^0}$ is $k$-dimensional by
\cref{thm:frame-profile-polynomial}.  There are only finitely many
isomorphism types arising from the smaller finite fields.  Adding them back
preserves $k$-dimensional asymptoticity by
\cref{lem:finite-perturbation}.  Minimality is
\cref{cor:pure-denominator-rigid}, and the forbidden-biclique assertion is
\cref{cor:polynomial-biclique}.
\end{proof}

\begin{example}
\label{ex:quadratic-graphs}
For $k=3$, the curve sort consists of the quadratic polynomials, and
$\abs{G_3(F)}=q^3+q^2$.
Choose one base point $(a,f_0(a))$.  The curves through it have the form
$$
  f_0(X)+(X-a)\bigl(\alpha(X-r)+\beta\bigr),
$$
so the associated affine slice is missing exactly the directions $a$ and
$\infty$.  Completing those two parallel classes and choosing the frame
described above
reconstructs the field.  Thus $\Gclass_3$ is three-dimensional and
$K_{3,n}$-free for every $n\geq2$.  A curve neighborhood has size $q$, so
it realizes the relative exponent $1/3$.
\end{example}

\begin{example}
\label{ex:cubic-graphs}
For $k=4$, choose two base points with distinct first coordinates
$a_1,a_2$.  The slice consists of
$
  f_0(X)+(X-a_1)(X-a_2)
  \bigl(\alpha(X-r)+\beta\bigr)
$
and is an affine plane missing the three directions
$a_1,a_2,\infty$.  The resulting pure incidence class is
four-dimensional and $K_{4,n}$-free for every $n\geq2$.
\end{example}

\begin{example}
\label{ex:small-degree-parameters}
For small $k$, the main parameters are:
$$
\begin{array}{c|c|c|c|c}
 k & \abs{C_k(F)} & \abs{G_k(F)} &
 \text{deleted slice directions} & \text{forbidden bicliques}\\ \hline
 2 & q^2 & 2q^2 & 1 & K_{2,n}\ (n\geq2)\\
 3 & q^3 & q^3+q^2 & 2 & K_{3,n}\ (n\geq2)\\
 4 & q^4 & q^4+q^2 & 3 & K_{4,n}\ (n\geq2)
\end{array}
$$
In each row the minimal asymptotic denominator is exactly $k$.
\end{example}

\begin{example}
For every $c\in C_k(F)$,
$
  \abs{\set{u:I_k(u,c)}}=q.
$
This formula detects the reduced exponent $1/k$.  It therefore forces the
minimal denominator asserted in \cref{cor:pure-denominator-rigid}.
\end{example}

\begin{example}
Let $u_1,\dots,u_r$ have pairwise distinct first-coordinate classes, with
$r\leq k$.  Then
$
  \abs{\set{c:\bigwedge_{i=1}^r I_k(u_i,c)}}=q^{k-r}.
$
These formulas realize exponents
$1/k,2/k,\dots,(k-1)/k$ on the curve sort.
\end{example}

\begin{example}
For fixed pairwise distinct curves $c_1,\dots,c_t$, we have
$\abs{\bigcup_{i=1}^t\set{u:I_k(u,c_i)}}=tq+O(1)$ because every pairwise
intersection has at most $k-1$ points.  Hence the
leading coefficient at point dimension $1/k$ can take every fixed positive
integer value.
\end{example}

\begin{remark}
The coordinate expansion in \cref{thm:coordinate-polynomial} gives the
asymptotic estimates. In \cref{thm:full-polynomial-incidence}, we also prove
definability of the parameter cells in the pure incidence language. For
this, a finite incidence frame gives field coordinates, and
\cref{thm:framed-profile-transfer} removes the frame parameters using a
target-definable profile partition. The reconstruction only uses a
two-dimensional slice of the $q^k$-element curve sort.
\end{remark}

\section*{Acknowledgements}
The ideas in this paper continue work begun in the author’s Ph.D. thesis. The author would like to thank his advisor, Camryn Hill, for first introducing him to this area of model theory.\\
\noindent\textbf{AI Use Statement.}
ChatGPT 5.5 was used for language editing, LaTeX formatting, and proof checking, and to assist with Lemma~\ref{lem:no-vector-in-pure-set} ,Theorems~\ref{thm:framed-polynomial-reconstruction}
and~\ref{thm:full-polynomial-incidence} and
Example~\ref{ex:small-degree-parameters}.
The author takes responsibility for the content of the manuscript.


\begin{thebibliography}{99}

\bibitem{AMSW2024}
S.~Anscombe, D.~Macpherson, C.~Steinhorn, and D.~Wolf,
\emph{Multidimensional asymptotic classes},
arXiv:2408.00102v1 [math.LO], 2024.

\bibitem{CDM1992}
Z.~Chatzidakis, L.~van den Dries, and A.~Macintyre,
\emph{Definable sets over finite fields},
J.\ Reine Angew.\ Math.~\textbf{427} (1992), 107--135.

\bibitem{ConantKruckman2019}
G.~Conant and A.~Kruckman,
\emph{Independence in generic incidence structures},
J.\ Symbolic Logic \textbf{84} (2019), no.~2, 750--780.

\bibitem{Dembowski1968}
P.~Dembowski,
\emph{Finite geometries},
Ergebnisse der Mathematik und ihrer Grenzgebiete, vol.~44,
Springer-Verlag, Berlin--Heidelberg--New York, 1968.

\bibitem{DelloStritto2011}
P.~Dello Stritto,
\emph{Asymptotic classes of finite Moufang polygons},
J.\ Algebra \textbf{332} (2011), 114--135.

\bibitem{Elwes2007}
R.~Elwes,
\emph{Asymptotic classes of finite structures},
J.\ Symbolic Logic \textbf{72} (2007), no.~2, 418--438.

\bibitem{ElwesMacpherson2008}
R.~Elwes and D.~Macpherson,
\emph{A survey of asymptotic classes and measurable structures},
in: Z.~Chatzidakis, D.~Macpherson, A.~Pillay, and A.~Wilkie (eds.),
\emph{Model Theory with Applications to Algebra and Analysis}, vol.~2,
London Math.\ Soc.\ Lecture Note Ser., vol.~350,
Cambridge Univ.\ Press, Cambridge, 2008, pp.~125--160.

\bibitem{GarciaMacphersonSteinhorn2015}
D.~Garc\'ia, D.~Macpherson, and C.~Steinhorn,
\emph{Pseudofinite structures and simplicity},
J.\ Math.\ Log.~\textbf{15} (2015), no.~1, 1550002.

\bibitem{HrushovskiWagner2008}
E.~Hrushovski and F.~O. Wagner,
\emph{Counting and dimensions},
in: Z.~Chatzidakis, D.~Macpherson, A.~Pillay, and A.~Wilkie (eds.),
\emph{Model Theory with Applications to Algebra and Analysis}, vol.~2,
London Math.\ Soc.\ Lecture Note Ser., vol.~350,
Cambridge Univ.\ Press, Cambridge, 2008, pp.~161--176.

\bibitem{LangWeil1954}
S.~Lang and A.~Weil,
\emph{Number of points of varieties in finite fields},
Amer.\ J.\ Math.~\textbf{76} (1954), no.~4, 819--827.

\bibitem{MacphersonSteinhorn2008}
D.~Macpherson and C.~Steinhorn,
\emph{One-dimensional asymptotic classes of finite structures},
Trans.\ Amer.\ Math.\ Soc.~\textbf{360} (2008), no.~1, 411--448.

\bibitem{Mirabi-MSMeas}
M.~Mirabi,
\emph{MS-measurability via Coordinatization},
arXiv:2109.11760v3 [math.LO], 2026.

\end{thebibliography}
\end{document}